\documentclass[noams]{compositio}

\usepackage{amsmath, amssymb}
\usepackage[hidelinks]{hyperref} 
\usepackage{enumerate}

\renewcommand{\contentsname}{Contents\\{\footnotesize\normalfont(A table
of contents should normally not be included)}}

\newtheorem{thm}{Theorem}[section]
\newtheorem{thma}{Theorem}

\newtheorem{prop}[thm]{Proposition}
\newtheorem{cor}[thm]{Corollary}
\newtheorem{lem}[thm]{Lemma}

\newcommand{\Prob}{\operatorname{Prob}} 
\newcommand{\BB}{\operatorname{\mathcal B}} 
\newcommand{\HH}{{\mathcal H}} 
\newcommand{\KK}{{\mathcal K}} 
\newcommand{\N}{{\mathbb N}} 

\theoremstyle{definition}
\newtheorem{defn}[thm]{Definition}
\newtheorem{examp}[thm]{Example}

\newcommand{\actson}{{\, \curvearrowright \,}}

\newcommand{\Har}{\operatorname{Har}}

\newcommand{\ovt}{\, \overline{\otimes}\,}
\newenvironment{alignLetter}{
	\setcounter{equation}{0}
	\renewcommand\theequation{\Alph{equation}}
	\align
}{
	\endalign
}

\begin{document}

\title{Poisson boundaries of II$_1$ factors}
\author{Sayan Das}
\email{sayan.das@ucr.edu}
\address{Department of Mathematics, University of California Riverside, 900 University Ave., Riverside, CA 92521, USA}
\author{Jesse Peterson}
\email{jesse.d.peterson@vanderbilt.edu}
\address{Department of Mathematics, Vanderbilt University, 1326 Stevenson Center, Nashville, TN 37240, USA}


\thanks{J.P. was supported in part by NSF Grant DMS \#1801125 and NSF FRG Grant \#1853989}

\begin{abstract}
We introduce Poisson boundaries of II$_1$ factors with respect to density operators that give the traces. The Poisson boundary is a von Neumann algebra that contains the II$_1$ factor and is a particular example of the boundary of a unital completely positive map as introduced by Izumi. Studying the inclusion of the II$_1$ factor into its boundary we develop a number of notions, such as double ergodicity and entropy, that can be seen as natural analogues of results regarding the Poisson boundaries introduced by Furstenberg. We use the techniques developed to answer a problem of Popa by showing that all finite factors satisfy the MV-property. We also extend a result of Nevo by showing that property (T) factors give rise to an entropy gap. 
\end{abstract}

\maketitle

\section{Introduction}

Given a locally compact group $G$ and a probability measure $\mu \in {\rm Prob}(G)$, the associated (left) random walk on $G$ is the Markov chain on $G$ whose transition probabilities are given by the measures $\mu * \delta_x$.  The Markov operator associated to this random walk is given by 
\[
\mathcal P_\mu(f)(x) = \int f(gx) \, d\mu(g),
\]
where $f$ is a continuous function on $G$ with compact support. The Markov operator extends to a contraction on $L^\infty(G)$, which is unital and (completely) positive. A function $f \in L^\infty(G)$ is $\mu$-harmonic if $\mathcal P_\mu(f) = f$. We let ${\rm Har}(G, \mu)$ denote the Banach space of $\mu$-harmonic functions. The Furstenberg-Poisson boundary \cite{furstenberg} of $G$ with respect to $\mu$ is a certain $G$-probability space $(B, \zeta)$, such that we have a natural positivity preserving isometric $G$-equivariant identification of $L^\infty(B, \zeta)$ with ${\rm Har}(G, \mu)$ via a Poisson transform. 

An actual construction of the Poisson boundary $(B, \zeta)$, which is often described as a quotient of the path space corresponding to the stationary $\sigma$-algebra, is less important to us here as its existence, and indeed, up to isomorphisms of $G$-spaces, it is the unique $G$-probability space such that $L^\infty(B, \zeta)$ is isomorphic, as an operator $G$-space, to ${\rm Har}(G, \mu)$. 

Under natural conditions on the measure $\mu$, the boundary $(B, \zeta)$ possesses a number of remarkable properties. It is an amenable $G$-space \cite{zimmeramenbound}, it is doubly ergodic with isometric coefficients \cite{kaimanovich1} \cite{glasnerweiss}, and it is strongly asymptotically transitive \cite{jaworski1, jaworski2}. The boundary has therefore become a powerful tool for studying rigidity properties for groups and their probability measure preserving actions \cite{margulissuperrigidity, zimmersuperrigidity, badershalom, burgermonod, baderfurman}. 

In light of the successful application of the Poisson boundary to rigidity properties in group theory, Alain Connes suggested (see \cite{jonesproblems}) that developing a theory of the Poisson boundary in the setting of operator algebras would be the first step toward studying his rigidity conjecture \cite{connesconjecture}, which states that two property (T) icc groups have isomorphic group von Neumann algebras if and only if the groups themselves are isomorphic. Further evidence for this can be seen by the significant role that Poisson boundaries play in \cite{creutzpetersonfree, creutzpetersonchar, petersonchar}, where a related rigidity conjecture of Connes was investigated. 

Poisson boundaries can more generally be defined using any Markov operator associated to a random walk. Markov operators are particular examples of normal unital completely positive (u.c.p.) maps on von Neumann algebras, and motivated by defining Poisson boundaries for discrete quantum groups, Izumi in \cite{izumi2, izumi4} was able to define a noncommutative Poisson boundary associated to any normal u.c.p.\ map on a general von Neumann algebra. Specifically, if $\mathcal M$ is a von Neumann algebra and $\phi: \mathcal M \to \mathcal M$ is a normal u.c.p.\ map, then we let ${\rm Har}(\phi) = \{ x \in \mathcal M \mid \phi(x) = x \}$ denote the space of $\phi$-harmonic operators. Izumi showed that there exists a (unique up to isomorphism) von Neumann algebra $\mathcal B_\phi$ such that, as operator systems, ${\rm Har}(\phi)$ and $\mathcal B_\phi$ can be identified via a Poisson transform $\mathcal P: \mathcal B_\phi \to {\rm Har}(\phi)$. The existence of this boundary follows by showing that ${\rm Har}(\phi)$ can be realized as the range of a u.c.p.\ idempotent on $\mathcal M$ and then applying a theorem of Choi and Effros. Alternatively, the existence of the boundary follows by considering the minimal dilation of $\phi$ \cite{izumi3}. We include in the appendix of this paper an elementary proof based on this perspective. 

There is a well-known dictionary between many analytic notions in group theory and those in von Neumann algebras. For example, states on $\mathcal B(L^2(M))$ correspond to states on $\ell^\infty \Gamma$, normal Hilbert $M$-bimodules correspond to unitary representations, etc., \cite[Section 2]{connessurvey} \cite{connescorrespondences}. This allows one to develop notions such as amenability, property (T), etc., in the setting of finite von Neumann algebras. While Izumi's boundary gives a satisfactory noncommutative analogue of the Poisson boundary associated to a general random walk, still missing is an appropriate notion of a noncommutative Poisson boundary analogous to the group setting. 

The main goal of this article is to introduce a theory of Poisson boundaries for finite von Neumann algebras that we believe will fill the role envisioned by Connes.  If $M$ is a finite von Neumann algebra with a normal faithful trace $\tau$, and if $\varphi \in \mathcal B(L^2(M, \tau))_*$ is a normal state such that $\varphi_{|M} = \tau$, then we will view $\varphi$ as the distribution of a ``noncommutative random walk'' on $M$.  To each distribution we associate a corresponding ``convolution operator'', which is a normal u.c.p.\ map $\mathcal P_\varphi: \mathcal B(L^2(M, \tau)) \to \mathcal B(L^2(M, \tau))$, such that $M \subset {\rm Har}(\mathcal P_\varphi)$. We then define the Poisson boundary of $M$ with respect to $\varphi$ to be Izumi's noncommutative boundary $\mathcal B_\varphi$ associated to $\mathcal P_\varphi$; more precisely the boundary is really the inclusion of von Neumann algebras $M \subset \mathcal B_\varphi$, together with the Poisson transform $\mathcal P: \mathcal B_\varphi \to {\rm Har}(\mathcal P_\varphi)$. 

Poisson boundaries of groups give rise to natural Poisson boundaries of group von Neumann algebras. Indeed, as was already noticed by Izumi in \cite{izumi3}, if $\Gamma$ is a countable discrete group and $\mu \in {\rm Prob}(\Gamma)$, then the noncommutative boundary of the u.c.p.\ map $\phi_\mu: \mathcal B(\ell^2 \Gamma) \to \mathcal B(\ell^2 \Gamma)$ given by $\phi_\mu(T) = \int \rho_\gamma T \rho_\gamma^* \, d\mu(\gamma)$ is naturally isomorphic to the von Neumann crossed-product $L^\infty(B, \zeta) \rtimes \Gamma$ where $(B, \zeta)$ is the Poisson boundary of $(\Gamma, \mu)$. Thus, many of the results we obtain are not merely analogues, but are actually generalizations of results from the theory of random walks on groups.

If $M$ is a finite factor, then under natural conditions on the distribution $\varphi$, e.g., that its ``support'' should generate $M$, we show that the boundary $\mathcal B_\varphi$ is amenable/injective (Proposition~\ref{prop:pbinjective}), and that the inclusion $M \subset \mathcal B_\varphi$ is ``ergodic'', i.e., $M' \cap \mathcal B_\varphi = \mathbb C$ (Proposition~\ref{prop:relcommutant}). We use techniques of Foguel \cite{foguel} to obtain equivalent characterizations for when the boundary is trivial (Theorem~\ref{thm:foguel}). The double ergodicity result of Kaimanovich \cite{kaimanovich1} is more subtle, as unlike in the case for groups, there is no natural ``diagonal'' inclusion of $M$ into $\mathcal B_\varphi \ovt \mathcal B_\varphi$. There is however a natural notions of left and right convolution operators, so that we may naturally associate with $\varphi$ a second u.c.p.\ map $\mathcal P^{\rm o}_\varphi$  which commutes with $\mathcal P_\varphi$ (see Section 3 for the precise definition of $\mathcal P^{\rm o}_\varphi$). We may then show that bi-harmonic operators are constant, a result which is equivalent to double ergodicity in the group setting.

\begin{thma}[Theorem \ref{thm:biharmconstant} below]
Let $M$ be a finite factor and suppose $\varphi$ is as above. Then we have
\[
{\rm Har}(\mathcal B(L^2(M, \tau)), \mathcal P_\varphi) \cap {\rm Har}(\mathcal B(L^2(M, \tau)), \mathcal P^{\rm o}_\varphi) = \mathbb C.
\] 
\end{thma}

Motivated by the question of determining whether or not $L\mathbb F_\infty$ is finitely generated, Popa studied in \cite{popa2018} the class of separable II$_1$ factors $M$ that are tight, i.e., $M$ contains two hyperfinite subfactors $L, R \subset M$ such that $L$ and $R^{\rm op}$ together generate $\mathcal B(L^2(M))$. He conjectures in Conjecture 5.1 of \cite{popa2018} that if a factor $M$ has the property that all amplifications $M^t$ are singly generated, then $M$ is tight. He also notes that a tight factor $M$ satisfies the MV-property, which states that for any operator $T \in \mathcal B(L^2(M))$ the weak closure of the convex hull of $\{ u (JvJ) T (Jv^*J) u^* \mid u, v \in \mathcal U(M) \}$ intersects the scalars. Popa then asks in Problem 7.4 of \cite{popaergodic} and Problem 6.3 in \cite{popatight} if free group factors, or perhaps all finite factors have the MV-property. As a consequence of double ergodicity we are able to answer Popa's problem.

\begin{thma}[Theorem \ref{popaqn19a} below]
All finite factors have the MV-property. 
\end{thma}

Other consequences of double ergodicity are that it allows us to show vanishing cohomology for sub-bimodules of the Poisson boundary (Theorem~\ref{innerderivation}), it allows us to generalize rigidity results from \cite{creutzpetersonchar} (Theorem~\ref{thm:boundaryrigidity}), and it allows us to extend results of Bader and Shalom \cite{badershalom} identifying the Poisson boundary of a tensor product with the tensor product of the Poisson boundaries (Corollary~\ref{cor:pbtensor}).

We also introduce analogues of Avez's asymptotic entropy and Furstenberg's $\mu$-entropy in the setting of von Neumann algebras (see Section~\ref{sec:entropy} for these definitions). We show that the triviality of the Poisson boundary is equivalent to the vanishing of the Furstenberg entropy (Corollary~\ref{cor:zero entropy1}). We also use entropy to extend a result of Nevo \cite{nevo} to the setting of von Neumann algebras, which shows that property (T) factors give rise to an ``entropy gap''.

\begin{thma}[Theorem \ref{thm:Tentropygap} below]
Let $M$ be a II$_1$ factor with property (T) generated by unitaries $u_1, \ldots, u_n$. Define the state $\varphi \in \mathcal B(L^2M)_*$ by $\varphi(T) = \frac{1}{n} \sum_{k = 1}^n \langle T \hat{u_k}, \hat{u_k} \rangle$. There exists $c > 0$ such that if $M \subsetneq \mathcal A$ is an irreducible inclusion of von Neumann algebras and $\zeta \in \mathcal A_*$ is any faithful normal state such that $\zeta_{| M} = \tau$, then $h_\varphi(M \subset \mathcal A, \zeta) \geq c$. 
\end{thma}

We end with an appendix where we construct Izumi's boundary of a u.c.p.\ map. Our approach is elementary, and has the advantage that it applies for general $C^*$-algebras. This level of generality has no doubt been known by experts, but we could not find this in the current literature.

\section{Boundaries}

\subsection{Hyperstates and bimodular u.c.p.\ maps}

Fix a tracial von Neumann algebra $(M, \tau)$, and suppose we have an embedding $M \subset \mathcal A$ where $\mathcal A$ is a $C^*$-algebra. We say a state $\varphi \in \mathcal A^*$ is a $\tau$-hyperstate (or just a hyperstate if $\tau$ is fixed) if it extends $\tau$. We denote by $\mathcal S_\tau(\mathcal A)$ the convex set of all hyperstates on $\mathcal A$. To each a hyperstate $\varphi$ we obtain a natural inclusion $L^2(M, \tau) \subset L^2(\mathcal A, \varphi)$ induced from the map $x \hat{1} \mapsto x 1_{\varphi}$ for $x \in M$. We let $e_M \in \mathcal B(L^2(\mathcal A, \varphi))$ denote the orthogonal projection onto $L^2(M, \tau)$. We may then consider the unital completely positive (u.c.p.) map $\mathcal P_\varphi: \mathcal A \to \mathcal B(L^2(M, \tau))$, defined by 
\begin{equation}\label{eq:poissontransform}
\mathcal P_\varphi(T) = e_M T e_M, \ \ \ \ \ \ \ T \in \mathcal A.
\end{equation}
Note that if $x \in M \subset \mathcal A$, then we have $\mathcal P_\varphi(x) = x$. We call the map $\mathcal P_\varphi$ the Poisson transform (with respect to $\varphi$) of the inclusion $M \subset \mathcal A$.

The following proposition is inspired from \cite[Section 2.2]{connessurvey}.

\begin{prop}\label{prop:hyperstatecorrespondence}
The correspondence $\varphi \mapsto \mathcal P_\varphi$ defined by (\ref{eq:poissontransform}) gives a bijective correspondence between hyperstates on $\mathcal A$, and u.c.p., $M$-bimodular maps from $\mathcal A$ to $\mathcal B(L^2(M, \tau))$. Moreover, if $\mathcal A$ is a von Neumann algebra, then $\mathcal P_\varphi$ is normal if and only if $\varphi$ is normal.

Also, this correspondence is a homeomorphism where the space of hyperstates is endowed with the weak$^*$-topology, and the space of u.c.p., $M$-bimodular maps with the topology of pointwise weak operator topology convergence.
\end{prop}
\begin{proof}
First note that if $\varphi$ is a hyperstate on $\mathcal A$, then for all $T \in \mathcal A$ we have 
$$
\varphi(T) 
= \langle T, \hat 1 \rangle_{\varphi} 
= \langle \mathcal P_\varphi(T) \hat 1, \hat 1 \rangle_\tau.
$$
From this it follows that the correspondence $\varphi \mapsto \mathcal P_\varphi$ is one-to-one. To see that it is onto, suppose that $\mathcal P: \mathcal A \to \mathcal B(L^2(M, \tau))$ is u.c.p.\ and $M$-bimodular.  We define a state $\varphi$ on $\mathcal A$ by $\varphi(T) = \langle \mathcal P(T) \hat{1}, \hat{1} \rangle_\tau$. For all $y \in M$ we then have $\varphi(y) = \langle \mathcal P(y) \hat 1, \hat 1 \rangle_\tau = \tau(y)$, hence $\varphi$ is a hyperstate. Moreover, if $y, z \in M$, and $T \in \mathcal A$, then we have
\begin{align}
\langle \mathcal P_\varphi(T) \hat y, \hat z \rangle_\tau
& = \langle \mathcal P_\varphi(z^* T y) \hat 1, \hat 1 \rangle_\tau \label{eq:state} \\ 
& = \varphi(z^* T y)
= \langle \mathcal P(T) \hat y, \hat z \rangle_\tau, \nonumber
\end{align}
hence, $\mathcal P_\varphi = \mathcal P$.

It is also easy to check that $\mathcal P_\varphi$ is normal if and only if $\varphi$ is.

To see that this correspondence is a homeomorphism when given the topologies above, suppose that $\varphi$ is a hyperstate, and ${\varphi_\alpha}$ is a net of hyperstates. From (\ref{eq:state}) and the fact that u.c.p.\ maps are contractions in norm we see that $\mathcal P_{{\varphi_\alpha}}$ converges in the pointwise ultraweak topology to $\mathcal P_\varphi$ if ${\varphi_\alpha}$ converges weak$^*$ to $\varphi$. Conversely, setting $y = z = 1$ in (\ref{eq:state}) shows that if $\mathcal P_{\varphi_\alpha}$ converges in the pointwise ultraweak topology to $\mathcal P_\varphi$, then $\varphi_\alpha$ converges weak$^*$ to $\varphi$. 
\end{proof}

Considering the case $\mathcal A = \mathcal B(L^2(M, \tau))$ we see that to each hyperstate $\varphi$ on $\mathcal B(L^2(M, \tau))$ we obtain a u.c.p.\ $M$-bimodular map $\mathcal P_\varphi$ on $\mathcal B(L^2(M, \tau))$. In particular, composing such maps gives a type of convolution operation on the space of hyperstates. More generally, if $\mathcal A$ is a $C^*$-algebra, with $M \subset \mathcal A$, then for hyperstates $\psi \in \mathcal A^*$, and $\varphi \in \mathcal B(L^2(M, \tau))^*$ we define the convolution $\varphi * \psi$ to be the unique hyperstate on $\mathcal A$ such that 
\begin{equation}\label{eq:convolution}
\mathcal P_{\varphi * \psi} = \mathcal P_{ \varphi} \circ \mathcal P_{ \psi}.
\end{equation}
We say that $\psi$ is $\varphi$-stationary if we have $\varphi * \psi = \psi$, or equivalently, if $\mathcal P_\psi$ maps into the space of $\mathcal P_\varphi$-harmonic operators
$$
{\rm Har}(\mathcal P_\varphi) = {\rm Har}( \mathcal B(L^2(M, \tau) ), \mathcal P_\varphi ) 
= \{ T \in \mathcal B(L^2(M, \tau) ) \mid \mathcal P_\varphi(T) = T \}.
$$

\begin{lem}
For a fixed $\psi \in \mathcal S_\tau( \mathcal A )$ the mapping 
$$
\mathcal S_\tau(\mathcal B(L^2(M, \tau))) \ni \varphi \mapsto  \varphi *  \psi \in \mathcal S_\tau(\mathcal A)
$$ 
is continuous in the weak$^*$-topology.

Moreover, if $\varphi \in \mathcal B(L^2(M, \tau))_*$ is a fixed normal hyperstate, then the mapping 
$$
\mathcal S_\tau(\mathcal A) \ni \psi \mapsto  \varphi *  \psi \in \mathcal S_\tau(\mathcal A)
$$ 
is also weak$^*$-continuous.
\end{lem}
\begin{proof}
By Proposition~\ref{prop:hyperstatecorrespondence} the correspondence $\varphi \mapsto \mathcal P_\varphi$ is a homeomorphism from the weak$^*$-topology to the topology of pointwise ultraweak convergence, this lemma then follows easily from (\ref{eq:convolution}).
\end{proof}

\subsection{Poisson boundaries of \texorpdfstring{{\rm II}{$_1$}}{II1} factors}

\begin{defn}
Let $\varphi \in  {\mathcal S}_\tau( \mathcal B(L^2(M, \tau) )$ be a hyperstate. We define the Poisson boundary $\mathcal B_\varphi$ of $M$ with respect to $\varphi$ to be the noncommutative Poisson boundary of the u.c.p.\ map ${\mathcal P}_\varphi$ as defined by Izumi \cite{izumi2}, i.e., the Poisson boundary $\mathcal B_\varphi$ is a $C^*$-algebra (a von Neumann algebra when $\varphi$ is normal) that is isomorphic, as an operator system, to the space of harmonic operators ${\rm Har}( \mathcal B(L^2(M, \tau) ), \mathcal P_\varphi )$.
\end{defn}
Since $M$ is in the multiplicative domain of ${\mathcal P}_\varphi$, we see that $\mathcal B_{\varphi}$ contains $M$ as a subalgebra. Moreover, note that if we have a $C^*$-algebra $\mathcal B$, an inclusion $M \subseteq \mathcal B$ together with a completely positive isometric surjection from $\mathcal B$ to ${\rm Har}( \mathcal B(L^2(M, \tau) ), \mathcal P_\varphi )$, then this induces a completely positive isometric surjection from $\mathcal B$ to $\mathcal B_{\varphi}$ which restricts to the identity on $M$. It's a well-known result of Choi \cite{choidissertation} that a completely positive surjective isometry between two $C^*$-algebras is a $\ast$-isomorphism. Thus, the Poisson boundary contains $M$ as a subalgebra, and the inclusion $(M \subset  \mathcal B_\varphi)$ is determined up to isomorphism by the property that there exists a completely positive isometric surjection $\mathcal P: \mathcal B_\varphi \to {\rm Har}( \mathcal B(L^2(M, \tau) ), \mathcal P_\varphi )$ which restricts to the identity map on $M$. We will always assume that $\mathcal P$ is fixed and we also call $\mathcal P$ the Poisson transform.

Given any initial hyperstate $\varphi_0 \in \mathcal S_\tau( \mathcal B(L^2(M, \tau) ) )$ we may consider the hyperstate given by $\varphi_0 \circ \mathcal P$ on $\mathcal B_\varphi$. Of particular interest is the state $\eta$ on $\mathcal B_\varphi$ arising from the initial hyperstate $\varphi_0(x) \in \mathcal S_{\tau}(\mathcal B(L^2(M, \tau)))$ given by $\varphi_0(x)= \langle x \hat 1, \hat 1 \rangle$, which we call the  stationary state on $\mathcal B_\varphi$. In this case, using (\ref{eq:state}) above, it is easy to see that we have $\mathcal P_\eta = \mathcal P$, and hence $\varphi * \eta = \eta$.

\begin{prop}\label{prop:pbinjective}
Let $(M, \tau)$ be a tracial von Neumann algebra and let $\varphi$ be a fixed hyperstate  on $\mathcal B(L^2(M, \tau))$. Then the Poisson boundary $\mathcal B_\varphi$ is injective.
\end{prop}
\begin{proof}

If we take any accumulation point $E$ of $\left\{ \frac{1}{N} \sum_{n = 1}^N \mathcal P_\varphi^n \right\}_{N \in \mathbb N}$ in the topology of pointwise ultraweak convergence, then $E: \mathcal B(L^2(M, \tau) ) \to {\rm Har}( \mathcal B(L^2(M, \tau) ), \mathcal P_\varphi )$ gives a u.c.p.\ projection. As $\mathcal B_\varphi$ is isomorphic to ${\rm Har}( \mathcal B(L^2(M, \tau) ), \mathcal P_\varphi )$ as an operator system it then follows that $\mathcal B_\varphi$ is injective \cite[Section 3]{choieffros}.
\end{proof}

The trivial case is when $\varphi_e(x) = \langle x \hat 1, \hat 1 \rangle_\tau$ in which case we have that $\mathcal P_{\varphi_e} = {\rm id}$, and the Poisson boundary is nothing but $\mathcal B(L^2(M, \tau))$. Note that $\varphi_e$ gives an identity with respect to convolution. Also note that if $\varphi \in \mathcal B(L^2(M, \tau))^*$ is a hyperstate, then we have a description of the space of harmonic operators as:
$$
{\rm Har}(\mathcal B(L^2(M, \tau)), \mathcal P_\varphi ) = \{ T \in \mathcal B(L^2(M, \tau)) \mid \varphi(a T b) =  \varphi_e(a T b) {\rm \ for \ all \ } a, b \in M \}.
$$

Since $\mathcal P_\varphi$ is $M$-bimodular it follows that $\mathcal P_\varphi(M') \subset M'$. We say that $\varphi$ is  regular if the restriction of $\mathcal P_\varphi$ to $M'$ preserves the canonical trace on $M'$, and we say that $\varphi$ is  generating if $M$ is the largest $*$-subalgebra of $\mathcal B(L^2(M, \tau))$ which is contained in ${\rm Har}(\mathcal B(L^2(M, \tau)), {\mathcal P_\varphi})$. 
If $\varphi$ is regular, then the  conjugate of $\varphi$ which is given by $\varphi^*(T) = \varphi(J T^* J)$, is again a hyperstate. We'll say that $\varphi$ is  symmetric if it is regular and we have $\varphi^* = \varphi$. 

Regular, generating, symmetric hyperstates are easy to find. Suppose $(M, \tau)$ is a separable finite von Neumann algebra with a faithful normal trace $\tau$. We consider the unit ball $(M)_1$ of $M$ as a Polish space endowed with the strong operator topology, and suppose we have a $\sigma$-finite measure $\mu$ on $(M)_1$ such that $\int x^*x \, d\mu(x) = 1$. We obtain a normal hyperstate as
\begin{equation}\label{eq:Poissonsum}
\varphi(T) = \int \langle T \widehat{x^*}, \widehat{x^*} \rangle \, d\mu(x)
\end{equation}
and using (\ref{eq:state}) we may explicitly compute the Poisson transform $\mathcal P_\varphi$ on $\mathcal B(L^2(M, \tau))$ as
$$
\mathcal P_\varphi(T) = \int ( J x^* J ) T (J x J) \, d\mu(x).
$$

\begin{prop}\label{prop:generatinghyperstates}
Consider $\varphi$ as given by (\ref{eq:Poissonsum}), then
\begin{enumerate}
\item $\varphi$ is generating if and only if the support of $\mu$ generates $M$ as a von Neumann algebra.
\item $\varphi$ is regular if and only if $\int x x^* \, d\mu(x) = 1$. In this case $\varphi^*$ is a normal hyperstate.
\item If $\varphi$ is regular, then $\mathcal P_{\varphi^*}(T) = \int (J x J)T(J x^* J) \, d\mu(x)$ and $\varphi$ is symmetric if $J_*\mu = \mu$, where $J$ is the adjoint operation. 
\end{enumerate}
\end{prop}
\begin{proof}
If the support of $\mu$ generates von Neumann algebra $M_0 \subset M$ such that $M_0 \not= M$, then we have $[JxJ, e_{M_0}] = [Jx^* J, e_{M_0} ] = 0$ for each $x$ in the support of $\mu$. Hence, $\mathcal P_\varphi(T) = \int (J x J) T (J x^* J) \, d\mu(x) = T$, for each $T$ in the $*$-algebra generated by $M$ and $e_{M_0}$. Therefore, $\varphi$ is not generating. On the other hand, if $T  \in {\rm Har}(\mathcal B(L^2(M, \tau)), \mathcal P_\varphi)$ is such that we also have $T^*T, TT^* \in {\rm Har}(\mathcal B(L^2(M, \tau)), \mathcal P_\varphi)$, then for each $a \in M$ we have 
\begin{align}
\int  \|  ( (J x J) T & - T (J x J) ) \hat{a} \|_2^2  \, d\mu(x) \nonumber \\
&= \langle (T^* \mathcal P_\varphi(1) T - \mathcal P_\varphi(T^*) T - T^* \mathcal P_\varphi(T) + \mathcal P_\varphi(T^*T) ) \hat{a}, \hat{a} \rangle
= 0, \nonumber
\end{align}
and by symmetry we also have $\int  \|  ( (J x J) T^*  - T^* (J x J) ) \hat{a} \|_2^2  \, d\mu(x)  = 0$.
Hence, $[Jx J, T] = [J x^* J, T] = 0$ for $\mu$-almost every $x \in (M)_1$. Therefore, if the support of $\mu$ generates $M$ as a von Neumann algebra, then $T \in JMJ' = M$, showing that $\varphi$ is generating, thereby proving (i). 

If $y \in M$, then we have $\mathcal P_\varphi(JyJ) = \int J x^* y x J \, d\mu(x)$. Hence, we see that $\varphi$ is regular if and only if for all $y \in M$ we have $\tau (y) = \int \tau (x^* y x) \, d\mu(x) = \int \tau(xx^* y) \, d\mu(x)$, which is if and only if $\int xx^* \, d\mu(x) = 1$, thereby proving (ii). 

If $\varphi$ is regular, then 
\begin{align}
\varphi^*(T) 
&= \varphi( J T^* J )
= \int \langle J T^* J \widehat{x^*}, \widehat{x^*} \rangle \, d\mu(x) \nonumber \\
&= \int \langle \hat{x}, T^* \hat{x} \rangle \, d\mu(x)
= \int \langle T \widehat{x^*}, \widehat{x^*} \rangle \, dJ_*\mu(x). \nonumber
\end{align}
Therefore, if $J_*\mu = \mu$, then $\varphi$ is symmetric, thereby proving (iii). 
\end{proof}

Given a unital $C^*$-algebra $A$, and a u.c.p.\ map $\mathcal P: A \to A$, we denote the set of fixed points of $\mathcal P$ by ${\rm Har}(A, \mathcal P)$. That is, ${\rm Har}(A, \mathcal P)=\{a \in A: \mathcal P(a)=a \}$. The following Lemma is well known, see, e.g., \cite{fnw},  \cite[Lemma 3.4]{bjkw}, or \cite[Lemma 3.1]{CD18} . We include a proof for the convenience of the reader.

\begin{lem}\label{lem:fixedalgebra}
Suppose $A$ is a unital $C^*$-algebra with a faithful state $\varphi$. If $\mathcal P: A \to A$ is a u.c.p.\ map such that $\varphi \circ \mathcal P = \varphi$, then ${\rm Har}( A, \mathcal P) \subset A$ is a $C^*$-subalgebra. 
\end{lem}
\begin{proof}
${\rm Har}( A, \mathcal P)$ is clearly a self-adjoint closed subspace, thus we must show that ${\rm Har}( A, \mathcal P)$ is an algebra. By the polarization identity it is enough to show that $x^* x \in {\rm Har}( A, \mathcal P)$ whenever $x  \in {\rm Har}( A, \mathcal P)$. Suppose  $x \in {\rm Har}(A, \mathcal P)$. By Kadison's inequality we have $\mathcal P(x^*x) - x^* x = \mathcal P(x^*x) - \mathcal P(x^*) \mathcal P(x) \geq 0$. Also, $\varphi( \mathcal P(x^*x) - x^*x) = 0$ so that by faithfulness of $\varphi$ we have $\mathcal P(x^*x) = x^*x$. 
\end{proof}

\begin{prop}\label{prop:relcommutant}
Let $M$ be a finite von Neumann algebra with a normal faithful trace $\tau$. Let $\varphi \in \mathcal B(L^2(M, \tau))^*$ be a regular generating hyperstate, and let $\mathcal B_\varphi$ be the corresponding Poisson boundary, then $M' \cap \mathcal B_\varphi = \mathcal Z(M)$. In particular, if $\varphi$ is also normal and $M$ is a factor, then $\mathcal B_\varphi$ is also a von Neumann factor. 
\end{prop}
\begin{proof}
Let $\mathcal P: \mathcal B_\varphi \to {\rm Har}(\mathcal B(L^2(M, \tau)), \mathcal P_\varphi )$ denote the Poisson transform. If $x \in M' \cap \mathcal B_\varphi$, then $\mathcal P(x) \in M' \cap \mathcal B(L^2(M, \tau)) = J M J$. Since $\varphi$ is regular, $\mathcal P_\varphi$ preserves the trace when restricted to $J M J$. Thus, ${\rm Har}(J M J,\mathcal P_\varphi)$ is a von Neumann subalgebra of $J M J$ by Lemma~\ref{lem:fixedalgebra}. 
 Since $\varphi$ is generating, $M$ is the largest von Neumann subalgebra of ${\rm Har}(\mathcal B(L^2(M, \tau))$, and hence ${\rm Har}(J M J,\mathcal P_\varphi) \subseteq M $, implying that ${\rm Har}(J M J,\mathcal P_\varphi)=\mathcal Z(M)$. Therefore, $\mathcal P(x) \in {\rm Har}(J M J, \mathcal P_\varphi) = \mathcal Z(M)$,
 and hence $x \in \mathcal Z(M)$ since $\mathcal P$ is injective.
\end{proof}

If $\varphi$ is a normal hyperstate in $\mathcal S_\tau( \mathcal B(L^2(M, \tau)))$, then $\mathcal P_\varphi: \mathcal B(L^2(M, \tau) ) \to \mathcal B(L^2(M, \tau))$ is a normal map, and hence the dual  map $\mathcal P_\varphi^*$ preserves the predual of $\mathcal B(L^2(M, \tau))$ which we identify with the space of trace-class operators. 

We let $A_\varphi  \in \mathcal B(L^2(M, \tau))$ denote the density operator associated with $\varphi$, i.e., $A_\varphi$ is the unique trace-class operator so that $\varphi(T) = {\rm Tr}(A_\varphi T)$ for all $T \in \mathcal B(L^2(M, \tau))$. Since $\varphi$ is positive we have that $A_\varphi$ is a positive operator. If $P_{\hat{1}}$ denotes the rank one orthogonal projection onto $\mathbb C \hat{1}$, then we have $\varphi(T) = \langle \mathcal P_\varphi(T) \hat{1}, \hat{1} \rangle = {\rm Tr}( \mathcal P_\varphi(T) P_{\hat{1}} )$, and hence we see that $A_\varphi = \mathcal P_\varphi^*( P_{\hat 1} )$. In particular we have that $A_{\varphi^{* n}} = (\mathcal P_{\varphi}^n)^*(P_{\hat{1}})$ for  $n \geq 1$.

\begin{prop} \label{prop:standard form}
Let $(M, \tau)$ be a tracial von Neumann algebra and let $\varphi \in \mathcal S_\tau( \mathcal B(L^2(M, \tau)))$ be a normal hyperstate, then there exists a $\tau$-orthogonal family $\{ z_n \}_n$ which gives a partition of the identity as $1 = \sum_n z_n^*z_n$ so that 
$$
\mathcal P_\varphi(T) = \sum_n (J z_n^* J) T (J z_n J)
$$ 
for all $T \in \mathcal B(L^2(M, \tau))$. 

Moreover, if $\{ \tilde z_m \}_m$ is a $\tau$-orthogonal family which gives a partition of the identity as $1 = \sum_n \tilde z_n^* \tilde z_n$, then the map $\sum_m (J \tilde z_m^* J) T (J \tilde z_m J)$ agrees with $\mathcal P_\varphi$ if and only if for each $t > 0$ we have
$$
{\rm sp} \{ z_n \mid \| z_n \|_2 = t \} = {\rm sp} \{ \tilde z_m \mid \| \tilde z_m \|_2 = t \}.
$$
\end{prop}
\begin{proof}
Since $A_\varphi$ is a positive trace-class operator we may write $A_\varphi = \sum_n a_n P_{y_n}$ where $a_1, a_2, \ldots $ are positive and $\{ y_n \}_n$ is an orthonormal family with $P_{y_n}$ denoting the rank one projection onto $\mathbb C y_n$. For $T \in \mathcal B(L^2(M, \tau))$ we then have
$$
{\rm Tr}(T A_\varphi)
= \sum_n a_n \langle T y_n, y_n \rangle.
$$
Taking $T = x^*x \in M$ we have $a_n \| x y_n \|_2^2 \leq {\rm Tr}(x^* x A_\varphi) = \| x \|_2^2$, so that  $y_n \in M \subset L^2(M, \tau)$ for each $n$. Hence, for $T \in \mathcal B(L^2(M, \tau))$ we have
\begin{align}
{\rm Tr}(\mathcal P_\varphi(T) P_{\hat{1}} )
&= {\rm Tr}(T A_\varphi) 
= \left\langle \sum_{n} a_n (J y_n J) T (J y_n^* J) \hat{1}, \hat{1} \right\rangle \nonumber \\
&= {\rm Tr} \left(   \left(\sum_{n} a_n (J y_n J) T  (J y_n^* J) \right) P_{\hat{1}} \right). \nonumber 
\end{align}

Since $\mathcal P_\varphi$ is $M$-bimodular and since $J y_n J \in M'$ it follows that for all $x, y \in M$ we have
$$
{\rm Tr}(\mathcal P_\varphi(T) x P_{\hat{1}} y) 
= {\rm Tr} \left(   \left(\sum_{n} a_n (J y_n J) T  (J y_n^* J) \right) x P_{\hat{1}} y \right).
$$
In particular, setting $T = y = 1$ we have
$$
\tau(x) = \sum_n a_n \tau(y_n^* y_n x), 
$$
which shows that $\sum_n a_n y_n^* y_n = 1$. 

Since the span of operators of the form $x P_{\hat{1}} y$ is dense in the space of trace-class operators it then follows that $\mathcal P_\varphi(T) = \sum_{n} a_n (J y_n J) T  (J y_n^* J)$ for all $T \in \mathcal B(L^2(M, \tau))$. Setting $z_n = \sqrt{a_n} y_n^*$ then finishes the existence part of the proposition.

Suppose now that $\{ \tilde z_m \}_m$ is a $\tau$-orthogonal family which gives a partition of the identity $1 = \sum_n \tilde z_n^* \tilde z_n$, and set $\tilde \varphi(T) = {\rm Tr}( ( \sum_n (J \tilde z_n^* J) T (J \tilde z_n J) ) P_{\hat{1}} )$. Then, the density matrix $A_{\tilde \varphi}$, corresponding to $\tilde \varphi$, is given by $A_{\tilde \varphi}=\sum_n \tilde z_n^* P_{\hat{1}} \tilde z_n$. Since $\{ \tilde z_n \}_n$ forms a $\tau$-orthogonal family it then follows easily that $\tilde z_n^*$ is an eigenvector for $A_{\tilde \varphi}$, and the corresponding eigenvalue is $\| \tilde z_n^* \|_2^2 = \| \tilde z_n \|_2^2$. 

Using our notation from the first part of the proof of the proposition, we have that $A_{ \varphi}=\sum_n  z_n^* P_{\hat{1}} z_n$. By the same argument as above, we get that $ z_n^*$ is an eigenvector for $A_{ \varphi}$, and the corresponding eigenvalue is $\|  z_n^* \|_2^2 = \| z_n \|_2^2$. Note that $\mathcal P_{\varphi}= \mathcal P_{\tilde \varphi}$ if and only if $A_{ \varphi}=A_{ \tilde \varphi}$. Since the corresponding density matrices are positive trace class operators, the moreover part of the proposition follow easily from the Spectral Theorem.
\end{proof}

We say that the form $\mathcal P_\varphi(T) = \sum_n (J z_n^* J) T (J z_n J)$ (resp.\ $\varphi(T) = \sum_n \langle T \widehat{z_n^*}, \widehat{z_n^*} \rangle$) is a  standard form for $\mathcal P_\varphi$ (resp.\ $\varphi$). It follows from Proposition~\ref{prop:generatinghyperstates} that $\varphi$ is generating if and only if $\{z_n \}_n$ generates $M$ as a von Neumann algebra. We say that $\varphi$ is strongly generating if the unital algebra (rather than the unital $*$-algebra) generated by $\{ z_n \}_n$ is already weakly dense in $M$. This is the case, for example, if $\varphi$ is generating and symmetric, since then we have that $\{ z_n \}_n = \{ z_n^* \}_n$, and hence the unital algebra generated by $\{ z_n \}_n$ is already a $*$-algebra.

\begin{prop}\label{prop:faithfulstate}
Let $(M, \tau)$ be a tracial von Neumann algebra and suppose $\varphi$ is a normal strongly generating hyperstate, then the stationary state $\zeta = \varphi \circ \mathcal P$ gives a normal faithful state on the Poisson boundary $\mathcal B_\varphi$ such that $\zeta_{| M } = \tau$. 
\end{prop}
\begin{proof}
By considering the Poisson transform $\mathcal P$, it suffices to show that $\varphi$ is normal and faithful on the operator system $\Har(\mathcal P_\varphi)$. Note that here the stationary state is a vector state and hence normality follows. To see that the state is faithful fix $T \in \Har(\mathcal P_\varphi)$, with $T \geq 0$ and $\langle T \hat{1}, \hat{1} \rangle = 0$. Let $\mathcal P_\varphi(S)= \sum_n (Jz_n^*J) S (J z_n J)$ be the standard form of $\mathcal P_\varphi$. Since $T \in \Har(\mathcal P_\varphi)$, we have that $\mathcal P_\varphi^k(T) =T$, for each $k \in \mathbb N$. Expanding the standard form gives
\[
0 = \langle T \hat{1}, \hat{1} \rangle
= \langle P_\varphi^k(T) \hat{1}, \hat{1} \rangle 
= \sum_{n_1, n_2, \ldots, n_k} \langle T z_{n_1} z_{n_2} \cdots z_{n_k} \hat{1}, z_{n_1} z_{n_2} \cdots z_{n_k} \hat{1} \rangle.
\]
We then have $T \hat{m}=0$ for all $m$ in the unital algebra generated by $\{ z_n \}$, and as $\varphi$ is strongly generating it then follows that $T=0$.
\end{proof}

We end this section by giving a condition for the boundary to be trivial. We denote the space of trace-class operators on $L^2(M, \tau)$ by ${\rm TC}(L^2(M, \tau))$. We also denote the trace-class norm on ${\rm TC}(L^2(M, \tau))$ by $\| \cdot \|_{\rm TC}$. We identify $\mathcal B(L^2(M, \tau))$ with ${\rm TC}(L^2(M, \tau))^{\ast}$ via the pairing $(A,T) \mapsto {\rm }Tr(AT)$, where $A \in {\rm TC}(L^2(M, \tau))$, and $T \in \mathcal B(L^2(M, \tau)). $

\begin{thm}\label{thm:foguel}
Let $(M, \tau)$ be a tracial von Neumann algebra and let $\psi$ be a normal hyperstate. set $\varphi = \frac{1}{2} \psi + \frac{1}{2} \langle \cdot \hat 1, \hat 1 \rangle$ and let $A_n \in {\rm TC}(L^2(M, \tau))$ denote the density matrix corresponding to the normal, u.c.p.\ $M$-bimodular map $\mathcal P_{\varphi}^n$. Then the following conditions are equivalent
\begin{enumerate}
\item For all $x \in M$ we have $\| x A_n - A_n x \|_{\rm TC} \to 0$. 
\item For all $x \in M$ we have $xA_n-A_nx \rightarrow 0$ weakly.
\item $\Har(\mathcal P_{\varphi})=M$ 
\end{enumerate}
\end{thm}

\begin{proof}
The first condition trivially implies the second. To see that the second implies the third suppose for each $x \in M$ we have $xA_n-A_nx \rightarrow 0$ weakly as $n \rightarrow \infty$. Let $T \in \Har(\mathcal P_{\varphi})$. Let $x, a,b \in M$. Then taking inner products in $L^2(M, \tau)$ we have
\begin{align*}
|\langle (TJxJ- JxJ T)  a\hat{1}, b\hat{1} \rangle| &= |\langle (b^*Tax^*-x^*b^*Ta) \hat{1}, \hat{1} \rangle | \\
&= | \langle \mathcal P_{\varphi}^n(b^*Tax^*-x^*b^*Ta) \hat{1}, \hat{1} \rangle | = |{\rm Tr}(A_n (b^*Tax^*-x^*b^*Ta) )| \\
&= |{\rm Tr}((x^*A_n-A_nx^*)b^*Ta)|  \rightarrow 0.
\end{align*}
Hence $T \in JMJ' = M$. 

To see that the third condition implies the first we adapt the approach of Foguel from \cite{foguel}. Suppose $\Har(\mathcal P_{\varphi})=M$. Set $\mathcal A_0 = \{ A \in {\rm TC}(L^2(M, \tau)) \mid \|(\mathcal P_{\varphi}^n)^*(A) \|_{\rm TC} \to 0 \}$.  Note that since $(\mathcal P_{\varphi}^n)^*$ is a contraction in the trace-class norm we have that $\mathcal A_0$ is a closed subspace. 

Since $\varphi = \frac{1}{2} \psi + \frac{1}{2} \langle \cdot \hat 1, \hat 1 \rangle$ we have $\mathcal P_\varphi^* = \frac{1}{2} {\rm id} + \frac{1}{2} \mathcal P_\psi^*$ and we compute 
\begin{align*}
(\mathcal P_\varphi^n)^* ({\rm id} - \mathcal P_\varphi^*)
&= 2^{-(n + 1)} \left( \sum_{k = 0}^n \binom{n}{k} (\mathcal P_\psi^k)^* \right) ({\rm id} - \mathcal P_\psi^*) \\
&= 2^{-(n + 1)} \sum_{k = 1}^n \left( \binom{n}{k-1} - \binom{n}{k}  \right) \mathcal P_{\psi}^*.
\end{align*}
We have $\lim_{n \to \infty} 2^{-(n + 1)} \sum_{k = 1}^n | \binom{n}{k-1} - \binom{n}{k}  | = 0$ (see (1.8) in \cite{ornsteinsucheston}) hence $\| (\mathcal P_\varphi^n)^*( P_{\hat 1} - \mathcal P_\varphi^*(P_{\hat 1}) )\|_{\rm TC} \to 0$. Thus $P_{\hat 1} - \mathcal P_{\varphi}^*(P_{\hat 1}) \in \mathcal A_0$.

Since $\mathcal P_{\varphi}^*$ is $M$-bimodular we then have that  $a P_{\hat 1} b - \mathcal P_{\varphi}^*(a P_{\hat 1} b) \in \mathcal A_0$ for each $a, b \in M$ and hence $B - \mathcal P_{\varphi}^*(B) \in \mathcal A_0$ for all $B \in {\rm TC}(L^2(M, \tau))$. If $T \in \mathcal B(L^2(M, \tau))$ is such that ${\rm Tr}(AT) = 0$ for all $A \in \mathcal A_0$, then for all $B \in {\rm TC}(L^2(M, \tau))$ we have $\langle B - \mathcal P_{\varphi}^*(B), T \rangle = 0$ so that $T \in {\rm Har}(\mathcal P_{\varphi}) = M$. Hence the annihilator of $\mathcal A_0$ is contained in $M.$ So the pre-annihilator of $M$ must be contained in $\mathcal A_0$.
Thus $A \in \mathcal A_0$ whenever ${\rm Tr}(A x) = 0$ for all $x \in M$. In particular, we have $x P_{\hat 1} - P_{\hat 1} x \in \mathcal A_0$ for all $x \in M$, which is equivalent to the fact that $\| x A_n - A_n x \|_{\rm TC} \to 0$ for each $x \in M$. 
\end{proof}

\section{Biharmonic operators}

If $\varphi \in \mathcal S_\tau(\mathcal B(L^2(M, \tau)))$ is regular and normal, then we define $\mathcal P_\varphi^{\rm o}$ to be the u.c.p.\ map given by $\mathcal P_\varphi^{\rm o} = {\rm Ad}(J) \circ \mathcal P_{\varphi^*} \circ {\rm Ad}(J)$. Note that $\mathcal P_\varphi^{\rm o}$ and $\mathcal P_\eta$ commute for any normal hyperstate $\eta$. Indeed, if we have standard forms $\mathcal P_\varphi(T) = \sum_n (J z_n^* J) T (J z_n J)$ and $\mathcal P_\eta(T) = \sum_m (J y_m^* J) T (J y_m J)$, then by Proposition~\ref{prop:generatinghyperstates} we have $\mathcal P_\varphi^{\rm o}(T) = \sum_n z_n T z_n^*$ and hence 
$$
\mathcal P_\varphi^{\rm o} \circ \mathcal P_\eta (T) = \mathcal P_\eta \circ \mathcal P_\varphi^{\rm o}(T) = \sum_{n, m} z_n (J y_m^* J) T (J y_m J) z_n^*.
$$

The following is a noncommutative analogue of double ergodicity which was established in \cite{kaimanovich1}.

\begin{thm} \label{thm:biharmconstant}
Let $(M, \tau)$ be a tracial von Neumann algebra and let $\varphi$ be a normal regular strongly generating hyperstate. Then 
$$
\Har(\mathcal B(L^2(M, \tau)), \mathcal P_\varphi) \cap \Har(\mathcal B(L^2(M, \tau)), \mathcal P_\varphi^{\rm o}) = \mathcal Z(M).
$$ 
\end{thm}
\begin{proof}
We fix a standard form $\mathcal P_\varphi(T) = \sum_n (J z_n^* J) T (J z_n J)$, so that we also have $\mathcal P_\varphi^{\rm o}(T) = \sum_m z_m T z_m^*$. We identify the Poisson boundary $\mathcal B_\varphi$ with $\Har(\mathcal B(L^2(M, \tau)), \mathcal P_{\varphi})$, and let $\zeta$ denote the stationary state on $\mathcal B_\varphi$, which is faithful by Proposition~\ref{prop:faithfulstate}. For $T \in \mathcal B_\varphi$ we have 
$$
\zeta( \mathcal P_\varphi^{\rm o}(T))
= \langle \mathcal P_\varphi^{\rm o}(T) \hat{1}, \hat{1} \rangle
= \langle \mathcal P_\varphi(T) \hat{1}, \hat{1} \rangle
= \zeta( \mathcal P_\varphi(T))
= \zeta(T).
$$

By Lemma~\ref{lem:fixedalgebra} we then have that $B_0 = \Har( \mathcal B_\varphi, \mathcal P^{\rm o}_{\varphi| \mathcal B_\varphi} )$ is a von Neumann subalgebra of $\mathcal B_\varphi$. If $p \in B_0$ is a projection and $\xi \in L^2(\mathcal B_\varphi, \zeta)$, then 
$$
\sum_n \| p z_n^* p^\perp \xi \|_2^2 
= \sum_n \langle z_n p z_n^* p^\perp \xi, p^\perp \xi \rangle 
= 0.
$$
We must therefore have $\| p z_n^* p^\perp \xi \|_2 = 0$ for each $n$, and hence $p z_n^* = p z_n^* p$, for each $n$. Repeating this argument with roles of $p$ and $p^\perp$ reversed shows that $z_n^* p = p z_n^*p$, so that $p \in M' \cap \mathcal B_\varphi$. Since $p$ was an arbitrary projection we then have $B_0 \subset M' \cap \mathcal B_\varphi$ and by Proposition~\ref{prop:relcommutant} we have $B_0=\mathcal Z(M)$.
\end{proof}

The previous result allows us to give an analogue of the classical Choquet-Deny theorem \cite{choquetdeny}, which states that if $\Gamma$ is an abelian group and $\mu \in {\rm Prob}(\Gamma)$ has support generating $\Gamma$, then every bounded $\mu$-harmonic function is constant.  

\begin{cor}[The Choquet-Deny theorem]
Suppose $M$ is an abelian von Neumann algebra and $\varphi$ is a normal regular strongly generating hyperstate, then 
$$
\Har(\mathcal B(L^2(M, \tau) ), \mathcal P_\varphi) = \mathcal Z(M) = M.
$$
\end{cor} 

We will now describe how Theorem~\ref{thm:biharmconstant} leads to a positive answer of a recent question by Popa \cite[Problem 7.4]{popaergodic} \cite[Problem 6.3]{popatight}. 

\begin{thm} \label{popaqn19a}
	Let $M$ be a finite von Neumann algebra with a normal faithful trace $\tau$ and let $\mathcal G \subset \mathcal U(M)$ be a group which generates $M$ as a von Neumann algebra. Then for any operator $T \in \mathcal B(L^2(M, \tau))$ the weak closure of the convex hull of $\{ u (J v J) T (J v^* J) u^* \mid u, v \in \mathcal G \}$ intersects $\mathcal Z(M)$. 
\end{thm}

\begin{proof} 
We first consider the case when $\mathcal G$ is countable. Let $\mu \in {\rm Prob}( \mathcal G)$ be symmetric with full support and define a normal regular symmetric generating hyperstate $\varphi$ by $\varphi(T) = \int \langle T \hat{u}, \hat{u} \rangle \, d\mu(u)$. The corresponding Poisson transform is then given by $\mathcal P_\varphi(T) = \int ( J u J) T (J u^* J) \, d\mu(u)$, and we may also compute $\mathcal P_\varphi^{\rm o}$ as $\mathcal P_\varphi^{\rm o}(T) = \int u^* T u \, d\mu(u)$. 

Fix $T \in \mathcal B(L^2(M, \tau))$ and let $\mathcal C= \overline{co}^{wk} \{ u (J v J) T (J v^* J) u^* \mid u, v \in \mathcal G \}$. Then 
$\mathcal C$ is preserved by both $\mathcal P_\varphi$ and $\mathcal P_\varphi^{\rm o}$ and hence $\mathcal C$ is preserved by any point-ultraweak limit points $E$ and $E^{\rm o}$ of $\left\{ \frac{1}{N} \sum_{n = 1}^N \mathcal P_\varphi^n \right\}_{N = 1}^\infty$ and $\left\{ \frac{1}{N} \sum_{n = 1}^N (\mathcal P_\varphi^{\rm o})^n \right\}_{N = 1}^\infty$ respectively. Since $\mathcal P_\varphi$ and $\mathcal P_\varphi^{\rm o}$ commute we have that $E$ and $E^{\rm o}$ commute. Moreover, as $\| \frac{1}{N} \sum_{n = 1}^N \mathcal P_\varphi^n - \frac{1}{N} \sum_{n = 1}^N \mathcal P_\varphi^{n + 1} \| \leq 2/N$ it follows that $E: \mathcal B(L^2(M, \tau)) \to {\rm Har}(\mathcal P_\varphi)$ and similarly $E^{\rm o}: \mathcal B(L^2(M, \tau)) \to {\rm Har}(\mathcal P_\varphi^{\rm o})$. By Theorem~\ref{thm:biharmconstant} we then have $E^{\rm o} \circ E: \mathcal B(L^2(M, \tau)) \to \mathcal Z(M)$. Hence
\[
E^{\rm o} \circ E(T) \in \mathcal C \cap \mathcal Z(M).
\]

\noindent In the general case, if $G < \mathcal G$ is a countable subgroup, then let $N \subset M$ be the von Neumann subalgebra generated by $G$ and let $e_N: L^2(M, \tau) \to L^2(N, \tau)$ be the orthogonal projection. If we define $\varphi$ as above and set $T_G = E^{\rm o} \circ E(T)$, then we have $T_G \in \mathcal C$, $e_N T_G e_N = E^{\rm o} \circ E( e_N T e_N)$ and viewing $e_N T e_N$ as an operator in $\mathcal B(L^2(N, \tau))$ we may apply Theorem~\ref{thm:biharmconstant} as above to conclude that $e_N T_G e_N \in \mathcal Z(N) \subset \mathcal B(L^2(N, \tau))$.
If we consider the net $\{ T_G \}_G \subset \mathcal B(L^2(M, \tau))$ where $G$ varies over all countable subgroups of $\mathcal G$, ordered by inclusion, then letting $T_0$ be any weak limit point of this net we have that $T_0 \in \mathcal C$. 

\noindent Fix $u \in \mathcal G$. Then for any countable subgroup $G< \mathcal G$, setting $N= G''$ and  $\tilde N= \langle G, u \rangle ''$, we have
$e_{\tilde N}[u, T_0]e_{\tilde N}= [u, e_{\tilde N} T_0 e_{\tilde N}]=0 $ and hence $e_N[u, T_0]e_N=0.$ If we consider the net of all countable subgroups $G< \mathcal G$ ordered by inclusion, then as $\mathcal G$ generates $M$, we have strong operator topology convergence $\lim\limits_{G \rightarrow \infty}e_{G''}=1.$ Hence it follows that $[u, T_0]=0$ and since $u \in \mathcal G$ was arbitrary, we have $T_0 \in \mathcal Z(M)$. 
\end{proof}

\noindent Let $(M, \tau)$ be a finite  von Neumann algebra and $T \in \mathcal B( L^2(M, \tau))$. Recall that the distance between $T$ and $\mathcal Z(M)$
is defined as ${\rm dist}(T,\mathcal Z(M))= {\rm inf} \{ ||T-S||: S \in \mathcal Z(M) \}$.  For $T \in \mathcal B( L^2(M, \tau))$ we let $\delta_T$ denote the derivation given by $\delta_T(x) = [x, T]$.

\begin{cor}
Let $M$ be a finite von Neumann algebra, and suppose $T \in \mathcal B(L^2(M))$, then 
$$
{\rm dist}(T, \mathcal Z(M)) \leq \| {\delta_T}_{ | M'} \| + \| {\delta_T}_{| M} \|.
$$
\end{cor}

\begin{proof}
This follows from the previous theorem since every point $S \in \{ u (JvJ) T (Jv^*J) u^* \mid u, v \in \mathcal U(M) \}$ satisfies ${\rm dist}(T, S) \leq \| {\delta_T}_{ | M'} \| + \| {\delta_T}_{| M} \|$.
\end{proof}

As another application of Theorem~\ref{thm:biharmconstant} we use Christensen's Theorem \cite[Theorem 5.3]{christensen1} to establish the following vanishing cohomology result; the case when $\mathcal C=M$ is the celebrated Kadison-Sakai Theorem \cite{kadisonder, sakaider}.
\begin{thm} \label{innerderivation}
Let $(M, \tau)$ be a tracial von Neumann algebra and let $\varphi$ be a normal regular strongly generating hyperstate, suppose $\mathcal C \subset \mathcal B_\varphi$ is a weakly closed $M$-bimodule. If $\delta: M \to \mathcal C$ is a norm continuous derivation, then there exists $c \in \mathcal C$ so that $\delta(x) = [x, c]$ for $x \in M$. Moreover, if $\varphi$ has the form $\varphi(T) = \int \langle T \widehat{u^*}, \widehat{u^*} \rangle \, d\mu(u)$ for some probability measure $\mu \in {\rm Prob}(\mathcal U(M))$, then $c$ may be chosen so that $\| c \| \leq \| \delta \|$.
\end{thm}
\begin{proof}
	Identifying $\mathcal C$ with its image under the Poisson transform we will view $\mathcal C$ as an operator system in $\Har(\mathcal P_\varphi) \subset \mathcal{B}(L^2(M, \tau))$. Since $L^2(M, \tau)$ has a cyclic vector for $M$, Christensen's Theorem \cite[Theorem 5.3]{christensen1} shows that $\delta(m)=mT-Tm$ for some $T \in \mathcal{B}(L^2(M, \tau))$. Taking the conditional expectation onto $\Har(\mathcal{P}_\varphi)$, we may assume $T \in \Har(\mathcal{P}_\varphi)$. 
	
	We suppose $\varphi$ is given in standard form $\varphi(T) = \sum_n \langle T \widehat{z_n^*}, \widehat{z_n^*} \rangle$. Note that $z_m\delta(z_m^*) \in \mathcal C$, so that
	\begin{align*}
	T - \mathcal P_\varphi^{\rm o}(T) = \sum_m z_m z_m^*T - \sum_m z_m Tz_m^* = \sum_m z_m \delta(z_m^*) \in \mathcal C
	\end{align*}
	As $\mathcal P_\varphi^{\rm o}$ leaves $\mathcal C$ invariant (since $\mathcal C$ is an $M$-bimodule), by induction we get that $T-(\mathcal P_\varphi^{\rm o})^n(T) \in \mathcal C$ for all $n \geq 1$, and hence for $N \geq 1$ we have
	\[
	T- \dfrac{1}{N} \sum_{n=1}^N (\mathcal P_\varphi^{\rm o})^n(T) \in \mathcal C
	\]
	If $z$ is a weak limit point of $\left\{ \dfrac{1}{N} \sum_{n=1}^N (\mathcal P_\varphi^{\rm o})^n(T) \right\}$, then $z \in \Har(\mathcal P_\varphi ^{\rm o}) \cap \Har(\mathcal P_\varphi)$ and so by Theorem~\ref{thm:biharmconstant} we have $z \in \mathcal{Z}(M)$. Thus, $T- z \in \mathcal C$ implements the derivation.
	
	For the moreover part, note that if $\varphi$ has the form $\varphi(T) = \int \langle T \widehat{u^*}, \widehat{u^*} \rangle \, d\mu(u)$ for some probability measure $\mu \in {\rm Prob}(\mathcal U(M))$, then 
\begin{align}
\|T-z\| 
&\leq \sup_N \|T- \frac{1}{N} \sum_{n=1}^N (\mathcal P_\varphi^{\rm o})^n(T)\|  \nonumber \\
&\leq \sup_n \| T - (\mathcal P_\varphi^{\rm o})^n(T) \| \nonumber \\
&= \sup_n \|  \int u \delta(u^*) \, d\mu_n \|
\leq \|\delta\|, \nonumber
\end{align}
where $\mu_n$ denotes the push forward of $\mu \times \mu \times \cdots \times \mu \in {\rm Prob} (\mathcal U(M)^n)$ under the multiplication map.

Hence $c=T-z$ implements $\delta$ with $\|c\| \leq \|\delta\|$.
\end{proof}

We remark that for a general hyperstate $\varphi$, in the proof of the previous theorem we still have $\| T - z \| \leq \| \delta \|_{cb}$, where $\| \delta \|_{cb}$ denotes the completely bounded norm of the derivation $\delta$ (see, for instance, \cite[Section 2]{christensen1} for the definition of the completely bounded norm). So that in general we may find $c \in \mathcal C$ with $\| c \| \leq \| \delta \|_{cb}$.

\section{Rigidity for u.c.p.\ maps on boundaries}
The main result in this section is Theorem \ref{thm:boundaryrigidity}, where we generalize \cite[Theorem 3.2]{creutzpetersonchar}. We mention several consequences, including a noncommutative version of \cite[Corollary 3.2]{badershalom}, which describes the Poisson boundary of a tensor product as the tensor product of Poisson boundaries.

\begin{thm}\label{thm:boundaryrigidity}
Let $(M, \tau)$ be a tracial von Neumann algebra, let $\varphi$ be a normal regular strongly generating hyperstate, and let $\mathcal B= \mathcal B_{\varphi}$ denote the corresponding boundary. Suppose we have a weakly closed operator system $\mathcal C$ such $M \subset \mathcal C \subset \mathcal B$. Let $\Psi: \mathcal C \to \mathcal B$ be a normal u.c.p.\ map such that $\Psi_{| M } = {\rm id}$. Then $\Psi = {\rm id}$.
\end{thm}
\begin{proof}
	Let $\mathcal P_\varphi(T) = \sum_n (J z_n^* J) T (J z_n J)$ denote the standard form of $\mathcal P_{\varphi}$ as in Proposition~\ref{prop:standard form}. Then by Proposition 2.4 we have $\mathcal P_\varphi^{\rm o}(T) = \sum_n z_n T z_n^*$.
	By identifying $\mathcal C$ with its image under the Poisson transform we may assume that $\mathcal C$ is a weakly closed $M$-subbimodule of $\Har(\mathcal{P}_{\varphi})$ and $\Psi: \mathcal C \to \Har(\mathcal{P}_{\varphi})$ is a normal u.c.p.\ map such that $\Psi_{| M } = {\rm id}$.
	Note that for $T \in \mathcal C$ we have,
\begin{align*}
\langle \Psi(T) \hat{1}, \hat{1} \rangle &=  \langle \mathcal P_{\varphi}(\Psi(T)) \hat{1}, \hat{1} \rangle =  \langle \mathcal P_{\varphi}^{\rm o}(\Psi(T)) \hat{1}, \hat{1} \rangle \\
&= \sum_n \langle z_n \Psi(T) z_n^* \hat{1}, \hat{1} \rangle = \langle \Psi(\mathcal P_{\varphi}^{\rm o}(T)) \hat{1}, \hat{1} \rangle,
\end{align*}
where the last equality follows from the fact that $\Psi$ is normal and $M$-bimodular, as $M$ is contained in the multiplicative domain of $\Psi$. Now, $\langle \Psi(\mathcal{P}_{\varphi}^{\rm o}(T)) \hat{1}, \hat{1} \rangle = \langle \Psi(T) \hat{1}, \hat{1} \rangle$ for all $T \in \mathcal C$ immediately implies that 
\[
\left\langle \Psi \left( \frac{1}{N}\sum_{n=1}^N (\mathcal{P}_{\varphi}^{\rm o})^n(T) \right) \hat{1}, \hat{1} \right\rangle = \langle \Psi(T) \hat{1}, \hat{1} \rangle \text{for all } T \in \mathcal C.
\]
Let $z$ be a weak operator topology limit point of $\frac{1}{N}\sum_{n=1}^N (\mathcal{P}_{\varphi}^{\rm o})^n(T)$. Then, $z \in \mathcal Z(M)$ by Theorem~\ref{thm:biharmconstant}, so that $\Psi(z)=z$. We then have
$$
\langle \Psi(T) \hat{1}, \hat{1} \rangle = \langle z \hat{1}, \hat{1} \rangle = \langle T \hat{1}, \hat{1} \rangle
$$
where the last equality follows because $z$ is independent of $\Psi$.
Now, let $a,b \in M$, and $T \in \mathcal C$. Then, we have that $b^*Ta \in \mathcal C$, and hence by above computation, we get
$$
\langle \Psi(T) a \hat{1}, b \hat{1} \rangle = \langle \Psi(b^*Ta) \hat{1}, \hat{1} \rangle =  \langle b^*Ta \hat{1}, \hat{1} \rangle =  \langle T a\hat{1}, b \hat{1} \rangle.
$$ 
Thus $\Psi(T)=T.$
\end{proof}

\begin{cor} \label{maxII}
	Let $M$ be a finite von Neumann algebra with a normal faithful trace $\tau$, and let $\varphi$ be a normal regular strongly generating hyperstate. Then, $M$ is a maximal finite von Neumann subalgebra inside $\mathcal B_{\varphi}.$
\end{cor}

\begin{proof}
	Suppose $N \subset \mathcal{B}_{\varphi}$ is a finite von Neumann algebra containing $M$. Then there exists a normal conditional expectation $E:N \to M$. Hence, by Theorem~\ref{thm:boundaryrigidity}, $E(x)=x$ for all $x \in N$, and hence $N=M$.
\end{proof}

\begin{cor} \label{boundarytypeIII}
		Let $M$ be a $\rm II_1$ factor, and let $\varphi$ be a normal regular strongly generating hyperstate. If $\mathcal B_\varphi \neq M$, then $\mathcal B_{\varphi}$ is a type $\rm III$ factor.
\end{cor}

\begin{proof}
	Note that the stationary state is normal and faithful by Proposition~\ref{prop:faithfulstate}, and $\mathcal B_\varphi$ is a factor by Proposition~\ref{prop:relcommutant}. We also note that Proposition ~\ref{prop:relcommutant} along with von Neumann's bicommutant Theorem shows that $\mathcal B_{\varphi}$ is not a type $\rm I$ factor.\\
	 Suppose $\mathcal B_\varphi$ is not a type III factor, then $\mathcal B_\varphi$ has a semi-finite normal faithful trace ${\rm Tr}$. As before, let $\mathcal P$ denote the Poisson transform, and let $\zeta$ be the normal state on $\mathcal B_{\varphi}$ defined by $\zeta(b)= \langle \mathcal P(b) \hat{1} , \hat{1} \rangle$. Fix $ 0 \leq T \in \mathcal B_{\varphi}$ with $Tr(T) < \infty$, and $\zeta(T) \neq 0$. Fix $S \in \mathcal B_{\varphi}$ with $S \geq 0$ and $Tr(S)< \infty$. 
Let $z$ be a ultraweak limit point of $\frac{1}{N}\sum_{n=1}^N (\mathcal{P}_{\varphi}^{\rm o})^n(T)$. Then by Theorem~\ref{thm:biharmconstant} we have $z \in \mathcal Z(M)= \mathbb C$ and arguing as in the proof of Theorem~\ref{thm:boundaryrigidity} we have $\zeta(T)=z.$ 
Therefore, $\zeta(T)Tr(S)$ is a limit point of $\{Tr((\frac{1}{N}\sum_{n=1}^N (\mathcal{P}_{\varphi}^{\rm o})^n(T) )S) \}_{N=1}^{\infty}.$ On the other hand, note that for each $N \in \mathbb N$, we have that $Tr(\frac{1}{N}\sum_{n=1}^N (\mathcal{P}_{\varphi}^{\rm o})^n(T)S)= Tr(T(\frac{1}{N}\sum_{n=1}^N (\mathcal{P}_{\varphi^*}^{\rm o})^n(S)) )$. Since $|Tr(T (\frac{1}{N}\sum_{n=1}^N (\mathcal{P}_{\varphi^*}^{\rm o})^n(S)))| \leq Tr(T) ||S||_{\infty}$, by the above discussion, we then have $$ \zeta(T)Tr(S) \leq Tr(T)||S||_{\infty}.$$ Consider a net of projections $\{S_i\}_{i \in I}$ in $\mathcal B_{\varphi}$, such that $S_i$ converges to $1$ in the strong operator topology. The above equation then shows that $\zeta(T)Tr(1) \leq Tr(T)< \infty$. As $\zeta(T) \neq 0$ by choice, we get that $Tr(1)< \infty$. Hence $\mathcal B_{\varphi}$ is a type $\rm II_1$ factor and by Corollary~\ref{maxII} we have that $\mathcal B_{\varphi}=M$. 
\end{proof}

\begin{thm} \label{thm:boundarytp}
Suppose for each $i \in \{ 1, 2 \}$, $M_i$ is a finite von Neumann algebra with normal faithful trace $\tau_i$. Let $\varphi_i$ and $\varphi_1 \otimes \varphi_2$ be normal regular strongly generating hyperstates for $M_i$ and $M_1 \bar{\otimes} M_2$ on $\mathcal B(L^2(M_i, \tau_i))$ and $\mathcal B(L^2(M_1 \bar{\otimes} M_2, \tau_1 \otimes \tau_2))$ respectively. Then,
$$
{\rm Har}(\mathcal P_{\varphi_1} \otimes \mathcal P_{\varphi_2}) = {\rm Har}(\mathcal P_{\varphi_1}) \ovt {\rm Har}(\mathcal P_{\varphi_2}).
$$
\end{thm}
\begin{proof}
We clearly have ${\rm Har}(\mathcal P_{\varphi_1}) \ovt {\rm Har}(\mathcal P_{\varphi_2}) \subset {\rm Har}(\mathcal P_{\varphi_1} \otimes \mathcal P_{\varphi_2})$ so we only need to show the reverse inclusion.  Note that 
$$
(\mathcal P_{\varphi_1} \otimes {\rm id}) \circ (\mathcal P_{\varphi_1} \otimes \mathcal P_{\varphi_2}) = (\mathcal P_{\varphi_1} \otimes \mathcal P_{\varphi_2}) \circ (\mathcal P_{\varphi_1} \otimes {\rm id}),
$$ 
hence $(\mathcal P_{\varphi_1} \otimes {\rm id})_{| {\rm Har}(\mathcal P_{\varphi_1} \otimes \mathcal P_{\varphi_2})}$ gives a normal ucp map which restricts to the identity on $M_1 \ovt M_2$. By Theorem~\ref{thm:boundaryrigidity} we have that $(\mathcal P_{\varphi_1} \otimes {\rm id})_{| {\rm Har}(\mathcal P_{\varphi_1} \otimes \mathcal P_{\varphi_2})}$ is the identity map and hence $$
{\rm Har}(\mathcal P_{\varphi_1} \otimes \mathcal P_{\varphi_2}) \subset {\rm Har}(\mathcal P_{\varphi_1} \otimes {\rm id}) = {\rm Har}(\mathcal P_{\varphi_1}) \ovt \mathcal B(L^2(M_2)).
$$ 
We similarly have 
$$
{\rm Har}(\mathcal P_{\varphi_1} \otimes \mathcal P_{\varphi_2}) \subset \mathcal B(L^2(M_1)) \ovt {\rm Har}(\mathcal P_{\varphi_2}).
$$

Since ${\rm Har}(\mathcal P_{\varphi_1})$ is injective it is semidiscrete \cite{connesinjective}, and hence has property $S_\sigma$ of Kraus \cite[Theorem 1.9]{krausslicemap}. We then have 
\begin{align}
{\rm Har}(\mathcal P_{\varphi_1} \otimes \mathcal P_{\varphi_2}) 
\subset ({\rm Har}(\mathcal P_{\varphi_1}) \ovt \mathcal B(L^2(M_2))) \cap (\mathcal B(L^2(M_1)) \ovt {\rm Har}(\mathcal P_{\varphi_2})) 
\subset {\rm Har}(\mathcal P_{\varphi_1}) \ovt {\rm Har}(\mathcal P_{\varphi_2}). \nonumber 
\end{align}
\end{proof}

\begin{cor}\label{cor:pbtensor}
Suppose for each $i \in \{ 1, 2 \}$, $M_i$ is a finite von Neumann algebra with normal faithful trace $\tau_i$. Let $\varphi_i$ and $\varphi_1 \otimes \varphi_2$ be normal regular strongly generating hyperstates for $M_i$ and $M_1 \bar{\otimes} M_2$ on $\mathcal B(L^2(M_i, \tau_i))$ and $\mathcal B(L^2(M_1 \bar{\otimes} M_2, \tau_1 \otimes \tau_2))$ respectively. Then, the identity map on $M_1 \ovt M_2$ uniquely extends to a $*$-isomorphism between $\mathcal B_{\varphi_1 \otimes \varphi_2}$ and $\mathcal B_{\varphi_1} \ovt \mathcal B_{\varphi_2}$.
\end{cor}

\section{Entropy}\label{sec:entropy}
In this section we introduce noncommutative analogues of Avez's asymptotic entropy \cite{avezentropy}, and Furstenberg entropy \cite[Section 8]{furstenbergproducts}.
\subsection{Asymptotic entropy}
Let $M$ be a tracial von Neumann algebra with a faithful normal tracial state $\tau$. For a normal hyperstate $\varphi \in \mathcal S_\tau(\mathcal B(L^2(M, \tau)))$ we define the {\bf entropy of $\varphi$}, denoted by $H(\varphi)$, to be the von Neumann entropy of the corresponding density matrix $A_{\varphi}$:
\begin{align*}
H(\varphi)= -{\rm Tr}(A_{\varphi}\log(A_{\varphi})).
\end{align*}
If we have a standard form $\varphi(T) = \sum_{n} \langle T \widehat{z_n^*}, \widehat{z_n^*} \rangle$, then we may compute this explicitly as 
$$
H(\varphi) = - \sum_n \| z_n \|_2^2 \log( \| z_n \|_2^2 ).
$$


\begin{thm}\label{thm:subadd}  If $\varphi$ and $\psi$ are two normal hyperstates with $\psi$ regular, then
	\begin{align*}
	H(\varphi \ast \psi) \leq H(\varphi)+ H(\psi)
	\end{align*}
\end{thm}
\begin{proof}
	Let $A_{\varphi}$ and $A_{\psi}$ be the corresponding density operators and $\mathcal{P}_{\varphi}$ and $\mathcal{P}_{\psi}$ be the corresponding u.c.p.\ $M$-bimodular maps. Suppose we have the standard forms 
	\begin{align*}
\varphi(T)= \sum_{i \in I} \langle T \mu_i^{1/2} \hat{a_i^*}, \mu_i^{1/2} \hat{a_i^*} \rangle
  \text{ with }\mu_i >0, \text{ }||a_i^*||_2=1, \text{ and } \tau(a_ja_i^*)=0 \text{ for all }i \neq j \in I. \\
\psi(T)= \sum_{j \in J} \langle T \nu_j \hat{c_j^*}, \nu_j \hat{c_j^*} \rangle  
 \text{ with }\nu_j >0, \text{ }||c_j^*||_2=1, \text{ and } \tau(c_kc_l^*)=0 \text{ for all }k \neq l \in J.
	\end{align*}

	 Hence $A_{\varphi}= \sum\limits_i \mu_i P_{\hat{a_i}}$ and $A_{\psi}= \sum\limits_j \nu_j P_{\hat{c_j}}$.
	
	Let $b_i=Ja_iJ$ and $d_i=Jc_iJ$ so that
	\begin{align*}
	\mathcal{P}_{\varphi}(T)=\sum\limits_i \mu_ib_iTb_i^{*} \text{ and } \mathcal{P}_{\psi}(T)=\sum\limits_j \nu_jd_j T d_j^{*}.
	\end{align*}
	Since $\psi$ is regular we have that $\sum_i \nu_i d_i^{*}d_i=\sum_i \nu_i d_id_i^{*}=1$. 
	Since $\varphi$ is a hyperstate we have that $\sum_i \mu_i b_ib_i^{*}=1$.
	Now,
	\[ 
	H(\varphi \ast \psi)=-\sum\limits_{i,j}Tr[\mu_i \nu_j b_i^{*}d_j^{*}P_{\hat{1}}d_jb_i \log(A_{\varphi \ast \psi})].
	\]
	and  
	\[
	b_i^{*}d_j^{*}P_{\hat{1}}d_jb_i=\tau(b_ib_i^{*}d_j^{*}d_j)P_{\widehat{b_i^{*}d_j^{*}}},
	\]
	so that for each $k, \ell$ we have 
	\[
	A_{\varphi \ast \psi}=\sum\limits_{i,j} \mu_i \nu_j b_i^{*}d_j^{*}P_{\hat{1}}d_jb_i \geq \mu_k \nu_\ell \tau(b_kb_k^{*}d_\ell^{*}d_\ell)P_{\widehat{b_k^{*}d_\ell^{*}}}.
	\]
	As $\log$ is operator monotone, for each $k, \ell$ we then have 
	\[
	-\log(A_{\varphi \ast \psi})=-\log(\sum\limits_{i,j}\mu_i \nu_j b_i^{*}d_j^{*}P_{\hat{1}}d_jb_i) \leq -\log((\mu_k \nu_\ell \tau(b_kb_k^{*}d_\ell^{*}d_\ell))P_{\widehat{b_k^{*}d_\ell^{*}}}).
	\] 
	Hence,
	\begin{align*}
	H(\varphi \ast \psi) 
	& \leq -\sum\limits_{i,j}Tr[\mu_i \nu_j\tau(b_ib_i^{*}d_j^{*}d_j) P_{\widehat{b_i^{*}d_j^{*}}}\log(\mu_i \nu_j\tau(b_ib_i^{*}d_j^{*}d_j) P_{\widehat{b_i^{*}d_j^{*}}})]\\
	&=  -\sum\limits_{i,j}Tr[\mu_i \nu_j\tau(b_ib_i^{*}d_j^{*}d_j)P_{\widehat{b_i^{*}d_j^{*}}}\log(\mu_i \nu_j\tau(b_ib_i^{*}d_j^{*}d_j))] \\
	           & \hskip .3in -\sum\limits_{i,j}Tr[\mu_i \nu_j\tau(b_ib_i^{*}d_j^{*}d_j) P_{\widehat{b_i^{*}d_j^{*}}}\log(P_{\widehat{b_i^{*}d_j^{*}}})]\\
	&=-\sum\limits_{i,j}\mu_i \nu_j\tau(b_ib_i^{*}d_j^{*}d_j) \log(\mu_i \nu_j\tau(b_ib_i^{*}d_j^{*}d_j)).
	\end{align*}
	
	Now define $m$ on $I \times J$ by $m(i,j)=\mu_i \nu_j\tau(b_ib_i^{*}d_j^{*}d_j)$.
	Note that 
	\[
	\sum\limits_im(i,j)=\nu_j \tau(\sum\limits_i \mu_i b_ib_i^{*}d_j^{*}d_j)=\nu_j\tau(d_j^{*}d_j)=\nu_j
	\] 
	and
	\[
	\sum\limits_jm(i,j)=\mu_i \tau(\sum\limits_i \nu_j b_ib_i^{*}d_j^{*}d_j)=\mu_i\tau(b_ib_i^{*})=\mu_i.
	\]
	To finish the proof it then suffices to show $$H(m)=- \sum\limits_{i,j}m(i,j)\log(m(i,j)) \leq H(\mu)+ H(\nu),$$
where $H(\mu)=-\sum\limits_i \mu_i \log(\mu_i)$ and $H(\nu)=-\sum\limits_i \nu_i \log(\nu_i)$. By the remark before Theorem ~\ref{thm:subadd}, a direct calculation yields $H(\mu)= H(\varphi)$ and $H(\nu)= H(\psi)$. 

\noindent Note that \begin{align*}
	H(m) &=- \sum\limits_{i,j}m(i,j)\log(m(i,j))\\
	&= - \sum\limits_{i,j} \mu_i \nu_j \tau( b_ib_i^{*}d_j^{*}d_j) \log(\mu_i\tau(b_ib_i^{*}d_j^{*}d_j)) -\sum\limits_{i,j} \mu_i \nu_j \tau( b_ib_i^{*}d_j^{*}d_j) \log(\nu_j)\\
	& = -\sum\limits_{i,j} \mu_i \nu_j \tau( b_ib_i^{*}d_j^{*}d_j) \log(\mu_i) \\
        & 	\hskip .3in -\sum\limits_{i,j} \mu_i \nu_j \tau( b_ib_i^{*}d_j^{*}d_j) \log(\nu_j) - \sum\limits_{i,j} \mu_i \nu_j \tau( b_ib_i^{*}d_j^{*}d_j) \log(\tau(b_ib_i^{*}d_j^{*}d_j)).
	\end{align*}
	In the last equality above, the first summation is $H(\mu)$, since summing over $j$ we get 
	\[
	-\sum\limits_i \mu_i \tau(b_ib_i^{*})\log(\mu_i)=-\sum\limits_i \mu_i \log(\mu_i),
	\] while the second summation is $H(\nu)$. Hence, all that remains is to show:
	\[
	\sum\limits_{i,j} \mu_i \nu_j \tau( b_ib_i^{*}d_j^{*}d_j) \log(\tau(b_ib_i^{*}d_j^{*}d_j)) \geq 0.
	\]
	
	Let $\eta(x)=-x\log(x)$ for $x \in [0,1]$. Note that $\eta$ is concave, and so $\eta(\sum\limits_{i}\alpha_ix_i) \geq \sum\limits_i \alpha_i \eta(x_i)$ whenever $\alpha_i \geq 0$ and $\sum\limits_i \alpha_i =1$.
	So,
	\begin{align*}
	-\sum\limits_{i,j} \mu_i \nu_j \tau( b_ib_i^{*}d_j^{*}d_j) \log(\tau(b_ib_i^{*}d_j^{*}d_j)) &= \sum\limits_{i,j} \mu_i \nu_j \eta(\tau( b_ib_i^{*}d_j^{*}d_j))\\
	& = \sum\limits_i \mu_i (\sum\limits_j \nu_j \eta(\tau( b_ib_i^{*}d_j^{*}d_j)))\\
	& \leq \sum\limits_i \mu_i  \eta(\sum\limits_j \nu_j \tau( b_ib_i^{*}d_j^{*}d_j))\\
	& = \sum\limits_i \mu_i \eta(\tau(b_ib_i^{*}))=0 
	\end{align*} 
\end{proof}

\begin{cor}
	If $\varphi$ is a normal regular hyperstate, then the limit $\lim\limits_{n \rightarrow \infty}\dfrac{H(\varphi^{*n})}{n}$ exists.
\end{cor}

\begin{proof}
	The sequence $\{H(\varphi^{*n})\}$ is subadditive by Theorem~\ref{thm:subadd} and hence the limit exists.
\end{proof}

The {\bf asymptotic entropy} $h(\varphi)$ of a normal regular hyperstate $\varphi$ is defined to be the limit 
$$
h(\varphi) = \lim\limits_{n \rightarrow \infty}\dfrac{H(\varphi^{*n})}{n}.
$$

\subsection{A Furstenberg type entropy}

Suppose $G$ is a Polish group and $\mu \in {\rm Prob}(G)$. Given a quasi-invariant action $G \overset{a}{\actson} (X, \nu)$ the corresponding Furstenberg entropy (or $\mu$-entropy) is defined \cite[Section 8]{furstenbergproducts} to be 
$$
h_\mu(a, \nu) = - \iint \log \left( \frac{d g^{-1} \nu}{d \nu} (x) \right) \, d\nu(x) d\mu(g).
$$

If we consider the measure space $(G \times X, \nu \times \mu)$, then we have a non-singular map $\pi: G \times X \to G \times X$ given by $\pi(g, x) = (g, g^{-1}x)$, whose Radon-Nikodym derivative is given by 
$$
\frac{d \pi(\mu \times \nu)}{d (\mu \times \nu)}(x, g) = \frac{dg^{-1}\nu}{d \nu}(x). 
$$
Recall that for  arbitrary positive functions $f,g \in L^1(X, \mu)$ (where $(X, \mu)$ is a standard probability space), the relative entropy of the measures $\mu_1=f d\mu$ and $\mu_2=g d\mu$, denoted by $S(\mu_1|\mu_2)$, is defined as $S(\mu_1|\mu_2)= \int_X f(\log(f)- \log (g))  d\mu$ (see \cite[Chapter 5]{petzbook}).
We may thus rewrite the $\mu$-entropy as a relative entropy
$$
h_\mu(a, \nu) 
= - \iint \log  \left( \frac{d \pi(\nu \times \mu)}{d(\nu \times \mu)}(g, x) \right) \, d(\nu \times \mu)
=S( (\nu \times \mu) | \pi(\nu \times \mu) ).
$$

Let $(M,\tau)$ be a tracial von Neumann algebra, $\varphi$ a normal hyperstate for $M$, and $\mathcal{A}$ a von Neumann algebra, such that $M \subseteq \mathcal{A}$. Let $\zeta \in \mathcal S_\tau(\mathcal{A})$ be a normal, faithful hyperstate. Let $\Delta_\zeta : L^2(\mathcal A, \zeta) \to L^2(\mathcal A, \zeta)$ be the modular operator corresponding to $\zeta$, and consider the spectral decomposition $\Delta_\zeta = \int_0^\infty \lambda \, dE(\lambda)$. We denote by $\Delta_n = \int_{1/n}^{n} \lambda d\lambda$, $n \geq 1,$ the truncations of the modular operator $\Delta$. We know that $\Delta_n$ converges to $\Delta$ in the resolvent sense.
 Throughout this section we denote the one parameter modular automorphism group associated with $\zeta$ by $\{ \sigma_t^{\zeta} \}_{t \in \mathbb R}.$ We also denote the corresponding modular conjugation operator by $J$, and let $S=J \Delta^{1/2}.$ We refer the reader to \cite[Chapters VI, VII, VIII]{Takesakibook} for details regarding Tomita-Takesaki Theory.

Since $\zeta|_{M}=\tau$, we have a natural inclusion of $L^2(M,\tau)$ in $L^2(\mathcal{A},\zeta)$. Let $e$ denote the orthogonal projection from $L^2(\mathcal{A},\zeta)$ to $L^2(M,\tau)$. The entropy of the inclusion $(M, \tau) \subset (\mathcal A, \zeta)$ with respect to $\varphi$ is defined to be 
$$
h_{\varphi}(M \subset \mathcal{A},\zeta) = - \int \log(\lambda) \, d\varphi(e E(\lambda) e).
$$
The next example shows that $h_{\varphi}(M \subset \mathcal{A},\zeta)$ can be considered as a generalization of the Furstenberg entropy.
\begin{examp}
	If $\Gamma$ is a discrete group, $\mu \in {\rm Prob}(\Gamma)$ and $\Gamma \overset{a}{\actson} (X, \nu)$ is a quasi-invariant action, then we may consider the state $\varphi$ on $\mathcal B(\ell^2 \Gamma)$ given by $\varphi(T) = \int \langle T \delta_\gamma, \delta_\gamma \rangle \, d\mu(\gamma)$, and we may consider the state $\zeta$ on $L^\infty(X, \nu) \rtimes \Gamma \subset \mathcal B(\ell^2\Gamma \ovt L^2(X, \nu))$ given by $\zeta \left(\sum_{\gamma \in \Gamma} a_\gamma u_\gamma \right) = \int a_e \, d\nu$. Note that a direct computation in this case yields $(\varphi * \zeta) \left(\sum_{\gamma \in \Gamma} a_\gamma u_\gamma \right) = \int a_e \, d (\mu*\nu)$. 
	The modular operator $\Delta_\zeta$ is then affiliated to the von Neumann algebra $\ell^\infty \Gamma \ovt L^\infty( X, \nu)$, and we may compute this directly as
	$$
	\Delta_\zeta(\gamma, x) = \frac{d\gamma^{-1} \nu}{d \nu}(x).
	$$
	We also have that the projection $e$ from $\ell^2\Gamma \ovt L^2(X, \nu) \to \ell^2 \Gamma$ is given by ${\rm id  } \otimes \int$. Thus, it follows that the measure $d\varphi(e E(\lambda) e)$ agrees with $d\alpha_*(\mu \times \nu)$, where $\alpha: \Gamma \times X \to \mathbb R_{> 0}$ is the Radon-Nikodym cocycle, $\alpha(\gamma, x) = \frac{d \gamma^{-1} \nu}{d \nu}(x)$.
	
	In this case we then have
	\begin{align}
	h_\varphi(L\Gamma \subset L^\infty(X, \nu) \rtimes \Gamma, \zeta) 
	& = -\int \log(\lambda) d \varphi(e E(\lambda) e) \nonumber \\
	&= - \iint \log \left( \frac{d \gamma^{-1} \nu}{d \nu} (x) \right) d(\nu \times \mu) 
	= h_\mu(a, \nu). \nonumber
	\end{align}
\end{examp}

\begin{lem}
	Let $\varphi \in \mathcal S_\tau(\mathcal B(L^2(M, \tau)))$ be a normal hyperstate and write $\varphi$ in a standard form $\varphi(T) = \sum_n \langle T \widehat{z_n^*}, \widehat{z_n^*} \rangle$. Suppose $\mathcal A$ is a von Neumann algebra with $M \subset \mathcal A$ and $\zeta \in \mathcal S_\tau(\mathcal A)$ is a normal hyperstate. Then if $h_\varphi(M \subset \mathcal A, \zeta) < \infty$ we have that $z_n^* 1_\zeta \in D( \log \Delta_\zeta)$ for each $n$ and 
	$$
	h_\varphi(M \subset \mathcal A, \zeta) 
	= -\sum_n \langle \log \Delta_\zeta z_n^* 1_\zeta, z_n^* 1_\zeta \rangle
	= i \lim_{t \to 0} \frac{1}{t} \sum_n (\zeta( z_n \sigma_t^\zeta(z_n^*) )-1).
	$$
\end{lem}

\begin{proof}
	As $\mathcal{A}1_{\zeta}$ forms a core for $S_{\zeta}$ we get that $z_n^*1_{\zeta} \in D(\log(\Delta_{\zeta}))$. Also, we know that $\lim\limits_{t \rightarrow 0}\dfrac{\Delta_{\zeta}^{it}-1}{t}\xi= i \log(\Delta_{\zeta})\xi$, for all $\xi \in D(\Delta_{\zeta})$. So, we have that
	\begin{align*}
	h_\varphi(M \subset \mathcal A, \zeta) &= -\varphi(e\log(\Delta_{\zeta})e) 
	= -\sum_n \langle \log \Delta_\zeta z_n^* 1_\zeta, z_n^* 1_\zeta \rangle \\
	&= i \sum_n \langle  z_n \lim\limits_{t \rightarrow 0}\dfrac{\Delta_{\zeta}^{it}-1}{t} z_n^* 1_{\zeta}, 1_{\zeta} \rangle
	= i \lim_{t \to 0} \frac{1}{t} \sum_n (\zeta( z_n \sigma_t^\zeta(z_n^*) )-1).
	\end{align*}
\end{proof}

\begin{examp}
	Fix two normal hyperstates $\varphi, \zeta \in \mathcal S_\tau( \mathcal B(L^2(M, \tau)))$ such that $\varphi$ is regular, and $\zeta$ is faithful, and consider the case $\mathcal A = \mathcal B(L^2(M, \tau))$. Then the density operator $A_\zeta$ is injective with dense range and the modular operator on $L^2( \mathcal B(L^2(M, \tau)), \zeta)$ is given by $\Delta_\zeta ( T 1_\zeta) = A_\zeta T A_\zeta^{-1} 1_\zeta$, for $T \in \mathcal B(L^2(M, \tau))$ such that $T 1_\zeta \in D(\Delta_\zeta)$. In particular note that $\log (\Delta_\zeta) (T 1_\zeta) = ( {\rm Ad}(\log A_\zeta) T ) 1_\zeta$, where ${\rm Ad}( \log A_\zeta) T = (\log A_\zeta) T - T (\log A_\zeta)$.
	
	We also have that the projection $e: L^2( \mathcal B(L^2(M, \tau)), \zeta) \to L^2(M, \tau)$ is given by $e( T 1_\zeta) = \mathcal P_\zeta(T) \hat{1}$. Therefore, $e \log \Delta_\zeta e x \hat{1} = \mathcal P_\zeta( {\rm Ad}(\log A_\zeta) x ) \hat{1} = \mathcal P_\zeta( {\rm Ad}(\log A_\zeta) ) x \hat{1}$. Hence,
	\begin{align}
	h_\varphi(M \subset \mathcal B(L^2(M, \tau)), \zeta)
	&= \varphi( \mathcal P_\zeta( {\rm Ad}(\log A_\zeta) ) ) \nonumber \\
	&= {\rm Tr}( A_{\varphi * \zeta} {\rm Ad}(\log A_\zeta) ) \nonumber \\
	&= {\rm Tr}( A_{\varphi * \zeta} \log A_\zeta ) - \langle \log A_\zeta \hat{1}, \hat{1} \rangle. \nonumber
	\end{align}
	Where the last equality follows since $\varphi$ is regular.
\end{examp}


We recall the following two lemmas from works of D.Petz \cite{Petz1986}.

\begin{lem}\label{lem;Petz1} 
Let $\Delta_j$ be positive, self adjoint operators on $\mathcal{H}_j, j=1,2.$ If $T: \mathcal{H}_1 \rightarrow \mathcal{H}_2$ is a bounded operator such that:
	\begin{itemize}
		\item $T (\mathcal{D}(\Delta_1)) \subseteq \mathcal{D}(\Delta_2)$
		\item $||\Delta_2T\xi|| \leq ||T||\cdot||\Delta_1 \xi||$ ($\xi \in \mathcal{D}(\Delta_1)$),
	\end{itemize}
	then we have for each $t \in [0,1]$, and $\xi \in \mathcal{D}(\Delta_1^t)$,
	\begin{align*}
	||\Delta_2^tT\xi|| \leq ||T||\cdot||\Delta_1^t\xi||
	\end{align*}
\end{lem}
\begin{lem} \label{lem:Petz2}
	Let $\Delta$ be a positive self adjoint operator and $\xi \in \mathcal{D}(\Delta)$. Then:
	\begin{align*}
	\lim\limits_{t \rightarrow 0+}\dfrac{||\Delta^{t/2}\xi||^2-||\xi||^2}{t}
	\end{align*}
\end{lem}
exists. It's finite or $- \infty$ and equals $\int\limits_{0}^{\infty}\log \lambda d\langle E_{\lambda} \xi , \xi \rangle$ where $\int\limits_{0}^{\infty}\log \lambda d E_{\lambda}$ is the spectral resolution of $\Delta$.

\begin{cor} \label{cor;eqdef1}
	$h_{\varphi}(M \subset \mathcal{A},\zeta)=-\lim\limits_{t \rightarrow 0+}\dfrac{\sum\limits_{k=1}^{\infty}||\Delta_{\zeta}^{t/2}ez_n^*\hat{1}||^2-||ez_n^*\hat{1}||^2}{t}$
\end{cor}

\begin{lem} \label{lem;nonneg}
	$h_{\varphi}(M \subset \mathcal{A}, \zeta) \geq 0$
\end{lem}

\begin{proof} 
	Let $\mathcal P_{\zeta}(T)=eTe$ for $T \in \mathcal{A}$. Let $\Delta_n = \int_{1/n}^{n} \lambda d\lambda$, $n \geq 1,$ denote the truncations of the modular operator $\Delta$.\\
	$h_{\varphi}(M \subset \mathcal{A}, \zeta)= \lim\limits_{n \rightarrow \infty}\varphi(-e \log\Delta_n e)=-\lim\limits_{n \rightarrow \infty}\langle \mathcal{P}_{\varphi} \circ \mathcal P_{\zeta} (\log \Delta_n) \hat{1} , \hat{1} \rangle \geq \lim\limits_{n \rightarrow \infty} -\langle \log (\mathcal{P}_{\varphi} \circ \mathcal P_{\zeta}  (\Delta_n)) \hat{1} , \hat{1} \rangle $ (using the operator Jensen's inequality; recall that $\log$ is operator concave).\\
	Notice that $e\Delta_n e \leq e\Delta e =e$. Since $\mathcal P_{\varphi}(e)=e$, we get $\mathcal{P}_{\varphi} \circ \mathcal P_{\zeta} (\Delta_n) \leq e \leq 1$. As $\log$ is operator monotone, we get that $\log (\mathcal{P}_{\varphi} \circ \mathcal P_{\zeta}  (\Delta_n)) \leq \log(1)=0$. Hence we are done.
\end{proof}

\begin{thm} \label{thm:entropadd}
		Let $\varphi, \psi \in \mathcal S_\tau( \mathcal B(L^2(M, \tau)))$ be two normal hyperstates such that $\psi$ is regular, and suppose $\mathcal A$ is a von Neumann algebra with $M \subset \mathcal A$, and $\zeta \in S_\tau(\mathcal A)$ is a normal, faithful hyperstate which is $\psi$-stationary. Then
	$$
	h_{\varphi * \psi}(M \subset \mathcal A, \zeta) 
	= h_{\varphi}(M \subset \mathcal A,\zeta) + h_{\psi}(M \subset \mathcal A, \zeta).
	$$
\end{thm}
\begin{proof}
	Suppose we have the standard forms 
	\begin{align*}
	\varphi(T)= \sum_{i \in I} \langle T \mu_i^{1/2} \hat{a_i^*}, \mu_i^{1/2} \hat{a_i^*} \rangle
	\text{ with }\mu_i >0, \text{ }||a_i^*||_2=1, \text{ and } \tau(a_ja_i^*)=0 \text{ for all }i \neq j \in I. \\
 \psi(T)= \sum_{j \in J} \langle T \nu_j \hat{b_j^*}, \nu_j \hat{b_j^*} \rangle  
	\text{ with }\nu_j >0, \text{ }||b_j^*||_2=1, \text{ and } \tau(b_kb_l^*)=0 \text{ for all }k \neq l \in J.
	\end{align*}
	Let $\mathcal{P}_{\varphi}$ and $\mathcal{P}_{\psi}$ be the corresponding u.c.p.\ maps so that $\mathcal{P}_{\varphi}(T)= \sum_k \mu_k Ja_k^*J T Ja_kJ$ and  $\mathcal{P}_{\psi}(T)= \sum_l \nu_l Jb_l^*J T Jb_lJ$.
	 We shall denote the projection from $L^2(\mathcal A, \zeta)$ to $L^2(M,\tau)$ by $e$ and $\Delta_{\zeta}$ by $\Delta$. We also denote the one parameter modular automorphism group corresponding to $\zeta$ by $\sigma_t$.
	  We then have
	\begin{align*}
	h_{\varphi}(M \subset \mathcal{A}, \zeta) &= i \lim\limits_{t \rightarrow 0} \varphi(\dfrac{e\Delta^{it}e-1}{t})= i \lim\limits_{t \rightarrow 0} \dfrac{1}{t}\varphi(e\Delta^{it}e-1) \\
	&= i \lim\limits_{t \rightarrow 0} \dfrac{1}{t}(\sum_k \mu_k \langle (\Delta^{it}-1 )a_k^* 1_\zeta, a_k^* 1_{\zeta} \rangle ) 
	\end{align*}
	Similarly,
	\begin{align*}
	h_{\psi}(M \subset \mathcal A, \zeta)=i \lim\limits_{t \rightarrow 0} \dfrac{1}{t}(\sum_l \nu_l \langle (\Delta^{it}-1 )b_l^* 1_\zeta, b_l^* 1_{\zeta} \rangle )  
	\end{align*}
	and,
	\begin{align*}
	h_{\varphi * \psi}(M \subset \mathcal A, \zeta) &= i \lim\limits_{t \rightarrow 0} \dfrac{1}{t}(\sum_{k,l} \mu_k \nu_l \langle (\Delta^{it}-1 )a_k^*b_l^* 1_\zeta, a_k^* b_l^*1_{\zeta} \rangle ) \\
	&= i \lim\limits_{t \rightarrow 0} \dfrac{1}{t}(\sum_{k,l} \mu_k \nu_l \langle (b_la_k\sigma_t(a_k^*b_l^*) 1_\zeta, 1_{\zeta} \rangle -1) 
	\end{align*}
	We shall now show: $ \lim\limits_{t \rightarrow 0} \dfrac{1}{t}(\sum_{k,l} \mu_k \nu_l \langle b_la_k\sigma_t(a_k^*b_l^*) 1_\zeta, 1_{\zeta} \rangle -\sum_{k,l} \mu_k \nu_l \langle b_l\sigma_t(b_l^*)\sigma_t(a_k^*) 1_\zeta, 1_{\zeta} \rangle ) =0.$ Let $y_t=a_k\sigma_t(a_k^*)$. Note that $y_t \rightarrow a_ka_k^*$ as $t \rightarrow 0$, in SOT.
	We have:
	\begin{align*}
	y_t\sigma_t(b_l^*)-\sigma_t(b_l^*)y_t &= y_t\sigma_t(b_l^*)-y_tb_l^*+y_t b_l^*-\sigma_t(b_l^*)y_t \\
	&= y_t (\sigma_t(b_l^*)- b_l^*)+ (y_t b_l^*-b_l^*y_t)+(b_l^*-\sigma_t(b_l^*))y_t
	\end{align*}
	Now,
	\begin{align*}
	\dfrac{1}{t}(\sum_{k,l} \mu_k \nu_l \langle (y_t b_l^*-b_l^*y_t) 1_{\zeta}, b_l^* 1_{\zeta} \rangle &=  \dfrac{1}{t}(\sum_{k,l} \mu_k \nu_l \langle b_ly_t b_l^* 1_{\zeta}, 1_{\zeta} \rangle - \dfrac{1}{t}(\sum_{k,l} \mu_k \nu_l \langle y_t 1_{\zeta}, b_lb_l^* 1_{\zeta} \rangle \\
	&= \dfrac{1}{t} \sum_k \mu_k \langle (\sum_l \nu_l b_l y_t b_l^*) 1_{\zeta}, 1_{\zeta} \rangle - \dfrac{1}{t} \sum_k \mu_k \langle y_t 1_{\zeta}, 1_{\zeta} \rangle\\
	&= \dfrac{1}{t} \langle y_t 1_{\zeta}, 1_{\zeta} \rangle- \dfrac{1}{t} \langle y_t 1_{\zeta}, 1_{\zeta} \rangle =0,
	\end{align*}
	where the second to last equality holds by $\psi$-stationarity of $\zeta$.\\
	Also, $ \lim\limits_{t \rightarrow 0} \dfrac{1}{t} (y_t (\sigma_t(b_l^*)-b_l^*))$ exists, and hence 
	\[ 
	\lim\limits_{t \rightarrow 0} \dfrac{1}{t}(\sum_{k,l} \mu_k \nu_l \langle b_la_k\sigma_t(a_k^*b_l^*) 1_\zeta, 1_{\zeta} \rangle -\sum_{k,l} \mu_k \nu_l \langle b_l\sigma_t(b_l^*)\sigma_t(a_k^*) 1_\zeta, 1_{\zeta} \rangle ) =0.
	\]
	So, we get that
	\begin{align*}
	h_{\varphi * \psi}(M \subset \mathcal A , \zeta )&= i \lim\limits_{t \rightarrow 0} \dfrac{1}{t}(\sum_{k,l} \mu_k \nu_l \langle (b_l\sigma_t(b_l^*)a_k\sigma_t(a_k^*)-1) 1_\zeta, 1_{\zeta} \rangle \\
	&=  i \lim\limits_{t \rightarrow 0} \dfrac{1}{t}(\sum_{k,l} \mu_k \nu_l[ \langle( b_l\sigma_t(b_l^*)-1) 1_{\zeta}, 1_{\zeta} \rangle \\
	&+ \langle (a_k\sigma_t(a_k^*)-1) 1_{\zeta}, 1_{\zeta} \rangle \\
	&+ \langle (a_k\sigma_t(a_k^*)-1) 1_{\zeta}, (b_l \sigma_t(b_l^*)-1)^* 1_{\zeta} \rangle ]
	\end{align*}
	The first term equals $h_{\varphi}(M \subset \mathcal A, \zeta)$, while second term equals $h_{\psi}(M \subset \mathcal A, \zeta)$, and the third term equals zero, as $\lim\limits_{t \rightarrow 0} \dfrac{1}{t} (a_k\sigma_t(a_k^*)-1) 1_{\zeta} $ exists, while $\lim\limits_{t \rightarrow 0} \sum_l \nu_l (b_l \sigma_t(b_l^*)-1)^* 1_{\zeta} =0.$
\end{proof}
\begin{cor} \label{cor:additivity under conv}
	Let $\varphi \in \mathcal S_\tau( \mathcal B(L^2(M, \tau)))$ be a regular normal hyperstate and suppose $\mathcal A$ is a von Neumann algebra with $M \subset \mathcal A$, and $\zeta \in S_\tau(\mathcal A)$ is a faithful $\varphi$-stationary hyperstate, then for $n \geq 1$ we have
	$$
	h_{\varphi^{*n}}(M \subset \mathcal A, \zeta) = n h_{\varphi}(M \subset \mathcal A, \zeta).
	$$
\end{cor}

\begin{lem} \label{lem;enineq1}
	$h_{\varphi}(M \subset \mathcal{A}, \zeta) \leq H(\varphi).$ 
\end{lem}

\begin{proof}
We continue the notation from the proof of Theorem~\ref{thm:entropadd}, so that $\mathcal P_{\varphi}(T)= \sum_k \mu_k b_kTb_k^*$. Let $a_k=Jb_kJ \in M$. 
	It follows from Lemma \ref{lem:Petz2} that
	\begin{align*}
	H(\varphi)=-\lim\limits_{t \rightarrow 0+}\dfrac{\sum\limits_{k=1}^{\infty}\mu_k||A_{\varphi}^{t/2}a_k^*\hat{1}||^2-||a_k^*\hat{1}||^2}{t}.
	\end{align*}
	So by Corollary \ref{cor;eqdef1} it's enough to show that
	\begin{align*}
	\lim\limits_{t \rightarrow 0+}\dfrac{\sum\limits_{k=1}^{\infty}\mu_k||A_{\varphi}^{t/2}a_k^*\hat{1}||^2-||a_k^*\hat{1}||^2}{t} \leq \lim\limits_{t \rightarrow 0+}\dfrac{\sum\limits_{k=1}^{\infty}\mu_k||\Delta_{\varphi}^{t/2}ea_k^*\hat{1}||^2-||ea_k^*\hat{1}||^2}{t}.
	\end{align*}	
So, it's enough to show that 
	\begin{align*}
	||A_{\varphi}^{t/2}a_k\hat{1}||^2 \leq ||\Delta_{\zeta}^{t/2}a_k1_{\zeta}||^2
	\end{align*}
	Define $T: L^2(\mathcal{A},\zeta) \rightarrow L^2(M,\tau)$ by $T(a1_{\zeta})=\mathcal P_{\zeta}(a)\hat{1}$. Then $||T|| =1$, as $||T(1_{\zeta})||=1$ and $||\mathcal P_{\zeta}|| \leq 1$.
	$T$ takes $\mathcal{D}(\Delta_{\zeta})$ into $\mathcal{D}(A_{\varphi})=L^2(M,\tau)$. We now denote $\Delta_{\zeta}$ by $\Delta$.
	By Lemma \ref{lem;Petz1} it's enough to show:
	\begin{align*}
	||A_{\varphi}^{1/2}T \xi|| \leq ||\Delta^{1/2} \xi|| \text{for all } \xi \in \mathcal{D}(\Delta).
	\end{align*}
	In fact it's enough to show the above for all vectors in a core for $\mathcal{D}(\Delta)$. Recall that $\mathcal{A}1_{\zeta}$ forms a core for $\mathcal{D}(\Delta)$. So, we only need to show 
	\begin{align*}
	||A_{\varphi}^{1/2}Ta1_{\zeta}|| \leq ||\Delta^{1/2}a1_{\zeta}|| \text { for all } a \in \mathcal A.
	\end{align*}
  To this end, let $a \in \mathcal A$. Recall that $S=J \Delta^{1/2}$, so that $\Delta^{1/2}=JS$. We then have
	\begin{align*}
	||\Delta^{1/2}a1_{\zeta}||^2 &= \langle \Delta^{1/2} a1_{\zeta}, \Delta^{1/2} a1_{\zeta} \rangle = \langle JS a 1_{\zeta}, JS a 1_{\zeta} \rangle \\
	&= \langle J a^* 1_{\zeta}, J a^* 1_{\zeta} \rangle = \langle a^* 1_{\zeta},a^* 1_{\zeta} \rangle = \zeta(aa^*) \\
	&= \langle \mathcal P_{\zeta}(aa^*) \hat{1}, \hat{1} \rangle
	\end{align*}
	We also have $\mathcal P_{\varphi} \circ \mathcal P_{\zeta} = \mathcal P_{\zeta} \implies \varphi \circ \mathcal P_{\zeta} = \zeta$.
	Now:
	\begin{align*}
	||A_{\varphi}^{1/2}Ta1_{\zeta}||^2 & = \langle A_{\varphi}^{1/2} \mathcal P_{\zeta}(a) \hat{1}, A_{\varphi}^{1/2} \mathcal P_{\zeta}(a) \hat{1} \rangle = \langle A_{\varphi} \mathcal P_{\zeta}(a) \hat{1}, \mathcal P_{\zeta}(a) \hat{1} \rangle \\
	&= \langle \mathcal P_{\zeta}(a)^* A_{\varphi} \mathcal P_{\zeta}(a) \hat{1}, \hat{1} \rangle \leq Tr(\mathcal P_{\zeta}(a)^*A_{\varphi}\mathcal P_{\zeta}(a))\\
	&= Tr(A_{\varphi}\mathcal P_{\zeta}(a)\mathcal P_{\zeta}(a^*)) \leq Tr(A_{\varphi}\mathcal P_{\zeta}(aa^*))\\
	&= \langle (\varphi \circ \mathcal P_{\zeta})(aa^*) \hat{1},\hat{1} \rangle = \langle \mathcal P_{\zeta}(aa^*) \hat{1}, \hat{1} \rangle \\
	&= \zeta(aa^*)=|| \Delta ^{1/2} a 1_{\zeta}||^2.
	\end{align*}
	Hence we are done.
\end{proof}
\begin{cor}\label{cor;enineq2}
	$h_{\varphi}(M \subset \mathcal{A}, \zeta) \leq h(\varphi)$
\end{cor}

\begin{proof}
	By Lemma \ref{lem;enineq1}, we have that $h_{\varphi^{*n}}(M \subset \mathcal{A}, \zeta) \leq H(\varphi^{*n})$. By Corollary \ref{cor:additivity under conv} we have that  $h_{\varphi^{*n}}(M \subset \mathcal{A}, \zeta) = n   h_{\varphi}(M \subset \mathcal{A}, \zeta) $. So we get,
	$$h_{\varphi}(M \subset \mathcal{A}, \zeta) \leq \dfrac{H(\varphi^{*n})}{n} \rightarrow h(\varphi).$$
\end{proof}

\begin{lem} \label{lem;entropyeq}
	$h_{\varphi}(M \subset \mathcal A, \zeta)=0$ if and only if there exists a normal $\zeta$ preserving  conditional expectation from $\mathcal A$ to $M$.
\end{lem}
\begin{proof}
Let  $\varphi$ be a standard form $\varphi(T) = \sum_k \langle T \widehat{a_k^*}, \widehat{a_k^*} \rangle$.	Let $\mathcal E: \mathcal A \rightarrow M$ be a normal $\zeta$ preserving conditional expectation. Then, we know that $\sigma^{\zeta}_t(m)=m$ for all $m \in M$, where $\sigma_t^{\zeta}$ denotes the modular automorphism group corresponding to $\zeta$. Hence,
	\begin{align*}
	h_{\varphi}(M \subset \mathcal A, \zeta)&= i \lim\limits_{t \rightarrow 0} \dfrac{1}{t} \sum_k \langle (\Delta^{it}-1) a_k^* 1_{\zeta}, a_k^* 1_{\zeta} \rangle \\
	&= i \lim\limits_{t \rightarrow 0} \dfrac{1}{t} \sum_k \langle \sigma_t(a_k^*) 1_{\zeta}, a_k^* 1_{\zeta} \rangle -1 =0.
	\end{align*}
	Conversely, suppose $h_{\varphi}(M \subset \mathcal A, \zeta)=0$. This part of the proof is motivated by the proof of Lemma 9.2 in \cite{petzbook}.
	Let $\Delta_{\zeta}=\Delta$ and let $\Delta = \int_{0}^{\infty} \lambda d\lambda$ be it's spectral resolution. Let $\Delta_n = \int_{1/n}^{n} \lambda d\lambda$, $n \geq 1$  be the truncations. We know that $\Delta_n$ converges to $\Delta$ in the resolvent sense. As usual, we denote by $e$ the projection from $L^2(\mathcal A, \zeta)$ to $L^2(M, \tau)$. We have that $e=e\Delta e \geq e \Delta_n e$ for all $n$. So, $(1+t)^{-1} \leq (e\Delta_n e+ t)^{-1} \leq e(\Delta_n +t)^{-1}e$ for all $n$ and for all $t>0$. Taking limits as $n \rightarrow \infty$, we get $(1+t)^{-1} \leq e(\Delta+t)^{-1}e$. Now we shall use the following integral representation of $\log$:
	$$
	\log(x)=\int_{0}^{\infty}[ (1+t)^{-1}-(x+t)^{-1}] dt
	$$
	So that
	$$h_{\varphi}(M \subset \mathcal A, \zeta)=-\int_{0}^{\infty}\sum_k \langle e[(1+t)^{-1}-(\Delta +t)^{-1}] e a_k^*\hat{1}, a_k^* \hat{1} \rangle dt.$$
	From $h_{\varphi}(M \subset \mathcal A, \zeta)=0$ and the above discussion, we deduce that
	\begin{align*}
&\langle e((1+t)^{-1}-(\Delta+t)^{-1})ea_k^* \hat 1, a_k^* \hat 1 \rangle =0 \\
&\Rightarrow e((1+t)^{-1}-(\Delta+t)^{-1})e a_k^* \hat 1=0 \\
&\Rightarrow (1+t)^{-1} a_k^* \hat 1= e(\Delta+t)^{-1} e a_k^* \hat 1
	\end{align*}for almost all $t>0$, and hence by continuity, for all $t>0$. We now show that the last relation also holds without the compression $e$. To this end, note that by differentiating the equation $(1+t)^{-1} a_k^* \hat 1= e(\Delta+t)^{-1}a_k^* \hat 1$ with respect to $t$, we get $(1+t)^{-2}a_k^* \hat 1= e(\Delta+t)^{-2}ea_k^* \hat 1$, for all $t>0$. Therefore, by the following norm calculation in $L^2(\mathcal A, \zeta)$ we have
\begin{align*}
	|| e(\Delta+t)^{-1}ea_k^* \hat 1||^2_2 &= ||(1+t)^{-1} a_k^* \hat 1||_2^2= \langle (1+t)^{-2}a_k^* \hat 1, a_k^* \hat 1 \rangle \\
&	= \langle e(\Delta+t)^{-2}ea_k^* \hat 1, a_k^*\hat 1 \rangle = \langle (\Delta+t)^{-2}a_k^* \hat 1, a_k^*\hat 1 \rangle =||(\Delta+t)^{-1}a_k^* \hat 1||^2_2.
\end{align*}
So we get that $(1+t)^{-1} a_k^* \hat 1= (\Delta+t)^{-1}  a_k^* \hat 1$ for all $t>0$.	This implies that $\Delta^{it}a_{k}^*1_{\zeta}=a_k^*1_{\zeta}$, which implies that $\sigma^{\zeta}_t(a_k^*)=a_k^*$ and hence $\sigma^{\zeta}_t(m)=m$ for all $m \in M$, as $\varphi$ is generating. Hence there exists a $\zeta$ preserving conditional expectation from $\mathcal A$ to $M$, which is normal, as $\zeta$ is normal.
\end{proof}
\begin{cor} \label{cor:zero entropy1}
	$\Har(\mathcal{B}(L^2(M, \tau)),\mathcal P_{\varphi})=M $ if and only if  $h_{\varphi}(M \subset \mathcal{B}_{\varphi}, \zeta)=0 $, where $\mathcal{B}_{\varphi}$ denotes the Poisson boundary with respect to $\varphi$.
\end{cor}

\begin{proof}
	If $h_{\varphi}(M \subset \mathcal{B_{\varphi}}, \zeta)=0 $, then by Lemma \ref{lem;entropyeq} there exists a normal conditional expectation $\mathcal{E}: \mathcal{B_{\varphi}} \rightarrow M$. By Theorem~\ref{thm:boundaryrigidity}, $\mathcal E= \rm {id}$, which implies that $\mathcal{B_{\varphi}}=M$, and hence 
	$$\Har(\mathcal P_{\varphi})= \mathcal{P}(\mathcal{B}_{\varphi})=\mathcal{P}(M)=M.$$
	Conversely, if $\Har(\mathcal{B}(L^2M, \tau),\mathcal P_{\varphi})=M$, then $\Delta_{\zeta}=I$ and hence $h_{\varphi}(M \subset \mathcal B_{\varphi}, \zeta)=0$ 
\end{proof}

\begin{cor} \label{cor:zero entropy2}
	$\Har(\mathcal{B}(L^2(M, \tau)),\mathcal P_{\varphi})=M $ if $h(\varphi)=0$.
\end{cor}

\begin{proof}
	Since $0 \leq h_{\varphi}(M \subset \mathcal{B}_{\varphi}, \zeta) \leq h(\varphi)$, this result follows from Corollary \ref{cor:zero entropy1}.
\end{proof}

\section{An entropy gap for property (T) factors}

If $(M, \tau)$ is a tracial von Neumann algebra, then a Hilbert $M$-bimodule consists of a Hilbert space $\mathcal H$, together with commuting normal representations $L: M \to \mathcal B(\mathcal H)$, $R: M^{\rm op} \to \mathcal B(\mathcal H)$. We will sometimes simplify notation by writing $x \xi y$ for the vector $L(x)R(y^{\rm op})\xi$. A vector $\xi \in \mathcal H$ is left (resp.\ right) tracial if $\langle x \xi, \xi \rangle = \tau(x)$ (resp.\ $\langle \xi x, \xi \rangle = \tau(x)$) for all $x \in M$. A vector is bi-tracial if it is both left and right tracial. 
A vector $\xi \in \mathcal H$ is central if $x \xi = \xi x$ for all $x \in M$. Note that if $\xi$ is a unit central vector, then $x \mapsto \langle x \xi, \xi \rangle$ gives a normal trace on $M$.

The von Neumann algebra $M$ has property (T) if for any sequence of Hilbert bimodules $\HH_n$, and $\xi_n \in \HH_n$ bi-tracial vectors, such that $\| x \xi_n - \xi_n x \| \to 0$ for all $x \in M$, then we have $\| \xi_n - P_0(\xi_n) \| \to 0$, where $P_0$ is the projection onto the space of central vectors. This is independent of the normal faithful trace $\tau$ \cite[Proposition 4.1]{popabettinumbers}. Property (T) was first introduced in the factor case by Connes and Jones \cite{connesjonesT} where they showed that for an ICC group $\Gamma$, the group von Neumann algebra $L\Gamma$ has property (T) if and only if $\Gamma$ has Kazhdan's property (T) \cite{kazhdanT}. Their proof works equally well in the general case when $\Gamma$ is not necessarily ICC.

We now suppose that $M$ is finitely generated as a von Neumann algebra. Take $\{ a_k \}_{k=1}^n \subset M$ a finite generating set such that $\sum_{k = 1}^n a_k^* a_k = \sum_{k = 1}^n a_ka_k^* = 1$, and let $\mathcal B(L^2(M, \tau)) \ni T \mapsto \varphi(T) = \sum_{k = 1}^n  \langle T \widehat{a_k^*}, \widehat{a_k^*} \rangle$ denote the associated normal regular hyperstate. For a fixed Hilbert bimodule $\mathcal H$ we define $\nabla_L, \nabla_R: \mathcal H \rightarrow \mathcal H^{\oplus n}$ by 
\begin{align*}
\nabla_L(\xi)=\oplus a_k\xi
\end{align*}
\begin{align*}
\nabla_R(\xi)=\oplus \xi a_k.
\end{align*}
Note that we have 
\[
\| \nabla_L(\xi) \|^2 
= \sum_{k = 1}^n \| a_k \xi \|^2 
= \left\langle \sum_{k = 1}^n a_k^*a_k \xi, \xi \right\rangle 
= \| \xi \|^2,
\] 
and we similarly have
\[
\| \nabla_R(\xi) \|^2
= \left\langle \sum_{k = 1}^n \xi a_k a_k^*, \xi \right\rangle 
= \| \xi \|^2.
\]
Thus $\nabla_L$ and $\nabla_R$ are both isometries. We let $T$ denote the operator given by $T \xi = \sum_{k = 1}^n a_k^* \xi a_k$. Note that $T= \nabla_L^* \nabla_R$ and hence $T$ is a contraction.

Suppose now that $M \subset \mathcal A$ is an inclusion of von Neumann algebras and $\zeta \in \mathcal A_*$ is a faithful normal hyperstate. We may then consider the Hilbert space $L^2(\mathcal A, \zeta)$ which is naturally a Hilbert $M$-bimodule where the left action is given by left multiplication $L(x) \hat{a} = \widehat{xa}$, and the right action is given by $R(x^{op}) = JL(x^*)J$. In this case the vector $\hat{1}$ is clearly left tracial, and we also have $J x^* J \hat{1} = \Delta^{1/2} x \hat{1}$ from which it follows that $\hat{1}$ is also right tracial. If $\xi_0 \in L^2(\mathcal A, \zeta)$ is a unit $M$-central vector, then $\tau_0(x) = \langle x \xi_0, \xi_0 \rangle$ defines a normal trace on $M$. We let $s \in \mathcal Z(M)$ denote the support of $\tau_0$.

\begin{lem}\label{lem:entropyineq}
Let $(M, \tau)$, $\varphi$, and $(\mathcal A, \zeta)$ be as given above, then 
\[
h_{\varphi}(M \subset \mathcal{A},\zeta) \geq -2\log \langle T 1_{\zeta}, 1_{\zeta} \rangle.
\]
\end{lem}

\begin{proof}
Let $\Delta = \int_{0}^{\infty} \lambda d\lambda$ be the spectral resolution of the modular operator and let $\Delta_m = \int_{1/m}^{m} \lambda d\lambda$, $m \geq 1$  be the truncations. Let $\mu_k=\tau(a_k^*a_k)$ and $b_k=\mu_k^{-1/2}a_k,$ for $k=1,2,\cdots, n.$ Note that $\sum_{k=1}^n \mu_k=1$.	Also note that $L_{a_k^*}R_{a_k}1_{\zeta}=a_k^*\Delta^{1/2}a_k1_{\zeta}$.
	Now,
\begin{align*}
-2\log \langle T1_{\zeta},1_{\zeta} \rangle 
&= -2\log (\sum_{k = 1}^n  \langle a_k^* \Delta^{1/2} a_k \hat{1}, \hat{1} \rangle) =-2\lim\limits_{m \rightarrow \infty}\log (\sum_{k = 1}^n \mu_k \langle b_k^* \Delta_m^{1/2} b_k \hat{1}, \hat{1} \rangle)\\
& \leq -2 \lim\limits_{m \rightarrow \infty}\sum_{k = 1}^n  \mu_k \log(\langle b_k^* \Delta_m^{1/2} b_k \hat{1}, \hat{1} \rangle) \leq -2 \lim\limits_{m \rightarrow \infty}\sum_{k = 1}^n  \mu_k \langle b_k^* \log(\Delta_m^{1/2}) b_k \hat{1}, \hat{1} \rangle \\
&= - \lim\limits_{m \rightarrow \infty}\sum_{k = 1}^n  \langle a_k^* \log(\Delta_m) a_k \hat{1}, \hat{1} \rangle 
= h_{\varphi}(M \subset \mathcal{A},\zeta),
\end{align*}
	where the second inequality follows from Jensen's operator inequality.
\end{proof}

\begin{thm}\label{thm:Tentropygap}
Let $M$ be a II$_1$ factor generated as a von Neumann algebra by $\{ a_k \}_{k = 1}^n$ such that $\sum_{k = 1}^n a_k^* a_k = \sum_{k = 1}^n a_k a_k^* = 1$. Let 
\[
\mathcal B(L^2(M, \tau)) \ni T \mapsto \varphi(T) = \sum_{k = 1}^n  \langle T \widehat{a_k^*}, \widehat{a_k^*} \rangle
\] 
denote the associated normal regular hyperstate. If $M$ has property (T), then there exists $c > 0$ such that if $M \subset \mathcal A$ is any irreducible inclusion having no normal conditional expectation from $\mathcal A$ to $M$, and if $\zeta \in \mathcal A_*$ is any faithful normal hyperstate, then $h_\varphi(M \subset \mathcal A, \zeta) \geq c$.
\end{thm}
\begin{proof}
Suppose $M$ has property (T) and there is a sequence of irreducible inclusions $M \subset \mathcal A_m$, and normal faithful hyperstates $\zeta_m \in \mathcal A_m$, such that $h_\varphi(M \subset \mathcal A_m, \zeta_m) \to 0$. Then by Lemma~\ref{lem:entropyineq}
we have that $\langle T 1_{\zeta_m}, 1_{\zeta_m} \rangle \to 1$, and hence $\sum_{k = 1}^n \| a_k 1_{\zeta_m} - 1_{\zeta_m} a_k \|_2^2 = 2 - 2 \langle T 1_{\zeta_m}, 1_{\zeta_m} \rangle \to 0$. Since $M$ has property (T) it then follows that for $m$ large enough there exists a unit $M$-central vector $\xi \in L^2(\mathcal A_m, \zeta_m)$. If we let $\tilde \zeta$ denote the state on $\mathcal A_m$ given by $\tilde \zeta(a) = \langle a \xi, \xi \rangle$, then as $\xi$ is $M$-central we have that $\tilde \zeta$ gives an $M$-hypertrace on $\mathcal A_m$. Thus, there exists a corresponding normal conditional expectation form $\mathcal A_m$ to $M$, for all $m$ large enough. 
\end{proof}


\section*{Acknowledgments} SD is immensely grateful to Darren Creutz for explaining the theory of Poisson boundaries of groups to him, and for many useful remarks and stimulating conversations about earlier drafts of this paper. SD would like to gratefully acknowledge many helpful conversations with Vaughan Jones, Ionut Chifan, Palle Jorgensen, and Paul Muhly regarding this paper. SD would also like to thank Ben Hayes and Krishnendu Khan for various discussions in and around the contents of this paper.  The authors would like to thank Sorin Popa for useful comments regarding this paper. The authors would like to thank the anonymous referee for numerous valuable comments and suggestions that greatly improved the exposition of the paper. 

\appendix
\section{Minimal dilations and boundaries of u.c.p.\ maps}\label{sec:appendix}

We include in this appendix a proof of Izumi's result from \cite{izumi2} that for a von Neumann algebra (or even an arbitrary $C^*$-algebra) $A$, and a u.c.p.\ map $\phi: A \to A$, the operator space ${\rm Har}(A, \phi)$ has a $C^*$-algebraic structure. We take the approach in \cite{izumi3} where ${\rm Har}(A, \phi)$ is shown to be completely isometric to the $*$-algebra of fixed points associated to a $*$-endomorphism which dilates the u.c.p.\ map.  There are several proofs of the existence of such a dilation, the first proof is by Bhat in \cite{bhatdilation} in the setting of completely positive semigroups, building on work from \cite{bhatindex}, \cite{bhatparthasarathy1}, and \cite{bhatparthasarathy2}, and then later proofs were given in \cite{bhatskeide}, \cite{muhlysolel}, and Chapter 8 of \cite{arvesonsemigroups}. Our reason for including an additional proof is that it is perhaps more elementary than previous proofs, being based on a simple idea of iterating the Stinespring dilation \cite{stinespring}.

\begin{lem}\label{lem:alggen}
If $\HH$ and $\KK$ are Hilbert spaces, and $V: \HH \to \KK$ is a partial isometry, then for $A \subset \BB(\HH)$, $B \subset \BB(\KK)$, we have that $V^*$ $*$-${\rm alg}(VBV^*, A) V = *$-${\rm alg} (B, V^*AV)$.
\end{lem}
\begin{proof}
Using the fact that $V^*V = 1$, this follows easily by induction on the length of alternating products for monomials in $VBV^*$, and $A$.
\end{proof}

If $A_0 \subset \BB(\HH_0)$ is a $C^*$-algebra, and $\phi: A_0 \to A_0$ is a unital completely positive map, then one can iterate Stinespring's dilation as follows:

\begin{lem} \label{lem:bhatdil}
Suppose $A_0 \subset \BB(\HH_0)$ is a unital $C^*$-algebra, and $\phi_0: A_0 \to A_0$ is a unital completely positive map. Then there exists a sequence whose entries consist of:
\begin{enumerate}[$($1$)$]
	
\itemsep-.05em 
\item a Hilbert space $\HH_n$;
\item an isometry $V_n: \HH_{n - 1} \to \HH_n$;
\item a unital $C^*$-algebra $A_n \subset \BB(\HH_n)$;
\item a unital representation $\pi_n: A_{n - 1} \to \BB(\HH_n)$, such that $\pi_n(A_{n - 1})$, and $V_n A_{n - 1} V_n^*$ generate $A_n$;
\item a unital completely positive map $\phi_n: A_n \to A_n$;
\end{enumerate}
such that the following relationships are satisfied for each $n \in \N$, $x \in A_{n - 1}$:
\begin{alignLetter} 
&  V_n^* \pi_n(x) V_n = \phi_{n - 1}(x); \label{item:stineA}  \\
&  V_n^* A_n V_n = A_{n - 1}; \label{item:stineB}  \\
&  \phi_n(\pi_n(x)) = \pi_n(\phi_{n - 1}(x)); \label{item:stineC} \\
&  \pi_{n + 1}(V_{n}xV_{n}^*) = V_{n+ 1} \pi_{n}(x) V_{n + 1}^*. \label{item:stineD} 
\end{alignLetter}

Moreover, for each $n \in \N$ we have that the central support of $V_nV_n^*$ in $A_n''$ is $1$. Also, if $A_0$ is a von Neumann algebra and $\phi_0$ is normal, then $A_n$ will also be a von Neumann algebra and $\pi_n$ and $\phi_n$ will be normal for each $n \in \N$.
\end{lem}
\begin{proof}
We will first construct the objects and show the relationships (\ref{item:stineA}), (\ref{item:stineB}), and (\ref{item:stineC}) by induction, with the base case being vacuous, and we will then show that (\ref{item:stineD}) also holds for all $n \in \N$. So suppose $n \in \N$ and that (\ref{item:stineA}), (\ref{item:stineB}), and (\ref{item:stineC}) hold for all $m < n$, (we leave $V_0$ undefined).

From the proof of Stinespring's Dilation Theorem we may construct a Hilbert space $\HH_n$ by separating and completing the vector space $A_{n - 1} \otimes \HH_{n - 1}$ with respect to the non-negative definite sesquilinear form satisfying
\[
\langle a \otimes \xi, b \otimes \eta \rangle = \langle \phi_{n -1} (b^* a) \xi, \eta \rangle,
\]
for all $a, b \in A_{n - 1}$, $\xi, \eta \in \HH_{n - 1}$.

We also obtain a partial isometry $V_n: \HH_{n - 1} \to \HH_n$ from the formula
\[
V_n(\xi) = 1 \otimes \xi,
\]
for $\xi \in \HH_{n - 1}$.

We obtain a representation $\pi_n: A_{n - 1} \to \BB(\HH_n)$ (which is normal when $A_0$ is a von Neumann algebra and $\phi_0$ is normal) from the formula
\[
\pi_n(x) (a \otimes \xi) = (xa) \otimes \xi,
\]
for $x, a \in A_{n - 1}$, $\xi \in \HH_{n - 1}$. And recall the fundamental relationship $V_n^* \pi_n(x) V_n = \phi_{n - 1}(x)$ for all $x \in A_{n - 1}$, which establishes (\ref{item:stineA}).

If we let $A_n$ be the $C^*$-algebra generated by $\pi_n(A_{n - 1})$ and $V_n A_{n - 1} V_n^*$, then $\pi_n : A_{n - 1} \to A_n$, and from Lemma~\ref{lem:alggen} we have that $V_n^* A_n V_n$ is generated by $V_n^* \pi_n(A_{n - 1}) V_n$ and $A_{n - 1}$. However, $V_n^* \pi_n (A_{n - 1}) V_n = \phi_{n - 1}(A_{n - 1}) \subset A_{n - 1}$, hence $V_n^* A_n V_n = A_{n - 1}$, establishing (\ref{item:stineB}). Also, when $A_0$ is a von Neumann algebra and $\pi_n$ is normal it then follows easily that $A_n$ is then also a von Neumann algebra.

Also note that $\pi_n(A_{n - 1}) V_nV_n^* \HH_n$ is dense in $\HH_n$, and so since $\pi_n(A_{n - 1}) \subset A_n$ we have that the central support of $V_nV_n^*$ in $A_n''$ is $1$.

We then define $\phi_n: A_n \to A_n$ by $\phi_n(x) = \pi_n(V_n^* x V_n)$, for $x \in A_n$. This is well defined since $V_n^* A_n V_n = A_{n - 1}$, unital, and completely positive. Note that for $x \in A_{n - 1}$ we have $\phi_n(\pi_n(x)) = \pi_n(V_n^* \pi_n(x) V_n) = \pi_n(\phi_{n - 1}(x))$, establishing (\ref{item:stineC}).

Having established (\ref{item:stineA}), (\ref{item:stineB}), and (\ref{item:stineC}) for all $n \in \N$, we now show that (\ref{item:stineD}) holds as well. 
For this, notice first that for $a, b \in A_{n}$, $x \in A_{n - 1}$, and $\xi, \eta \in \HH_n$ we have 
\begin{align}
\langle \pi_{n + 1}(V_n x V_n^* ) (a \otimes \xi), b \otimes \eta \rangle 
& = \langle V_n x V_n^*a \otimes \xi, b \otimes \eta \rangle \nonumber \\
& = \langle \phi_n(b^* V_n x V_n^*a) \xi, \eta \rangle \nonumber \\
& = \langle \pi_n(V_n^* b^* V_n x V_n^* a V_n) \xi, \eta \rangle \nonumber \\
& = \langle 1 \otimes \pi_n(x V_n^* a V_n) \xi, b \otimes \eta \rangle. \nonumber
\end{align}

Setting $x = 1$ and using that $V_{n + 1}^*( 1 \otimes \zeta) = \zeta$ for each $\zeta \in \HH_n$, we see that 
\begin{align}
(V_{n+1}V_{n+1}^*)\pi_{n + 1}(V_nV_n^*) ( a \otimes \xi) 
&= (V_{n+1}V_{n+1}^*)( 1 \otimes \pi_n(V_n^*aV_n)\xi ) \nonumber \\
&= 1 \otimes \pi_n(V_n^*aV_n)\xi \nonumber \\
&= \pi_{n + 1}(V_nV_n^*) ( a \otimes \xi), \nonumber
\end{align}
and hence $\pi_{n + 1}(V_nV_n^*) \leq V_{n + 1} V_{n + 1}^*$. If instead we set $a = 1$, then we have
\begin{align}
V_{n + 1} \pi_n(x) \xi = 1 \otimes \pi_n(x) \xi = \pi_{n + 1}(V_n x V_n^* ) V_{n + 1}  \xi, \nonumber
\end{align}
and so $V_{n + 1} \pi_n(x) = \pi_{n + 1}(V_n x V_n^* ) V_{n + 1}$. Multiplying on the right by $V_{n + 1}^*$ and using that $\pi_n(V_nV_n^*) \leq V_{n + 1} V_{n + 1}^*$ then gives $V_{n + 1} \pi_n(x) V_{n + 1}^* = \pi_{n + 1}(V_n x V_n^*)$.
\end{proof}

\begin{thm}[Bhat \cite{bhatdilation}]\index{Bhat's Dilation Theorem}
Let $A_0 \subset \BB(\HH_0)$ be a unital $C^*$-algebra, and $\phi_0: A_0 \to A_0$ a unital completely positive map. Then there exists 
\begin{enumerate}[$($1$)$]
\itemsep-.05em
\item a Hilbert space $\KK$;
\item an isometry $W: \HH_0 \to \KK$;
\item a $C^*$-algebra $B \subset \BB(\KK)$; 
\item a unital $*$-endomorphism $\alpha: B \to B$;
\end{enumerate}
such that $W^* B W = A_0$, and for all $x \in A_0$ we have
$$
\phi_0^k(x) = W^* \alpha^k(W x W^*) W.
$$

Moreover, we have that the central support of $WW^*$ in $B''$ is $1$, $\alpha^{k}(WW^*) \leq \alpha^{k+1}(WW^*)$, and for $y \in \BB(\KK)$ we have $y \in B$ if and only if $\alpha^k(WW^*) y \alpha^k(WW^*) \in \alpha^k(WA_0W^*)$ for all $k \geq 0$. Also, if $A_0$ is a von Neumann algebra and $\phi_0$ is normal, then $B$ will also be a von Neumann algebra, and $\alpha$ will also be normal.
\end{thm}
\begin{proof}
Using the notation from Lemma~\ref{lem:bhatdil}, we may define a Hilbert space $\KK$ as the directed limit of the Hilbert spaces $\HH_n$ with respect to the inclusions $V_{n + 1} : \HH_n \to \HH_{n + 1}$. We denote by $W_n: \HH_n \to \KK$ the associated sequence of isometries satisfying $W_{n + 1}^* W_n = V_{n + 1}$, for $n \in \N$, and we set $P_n = W_nW_n^*$, an increasing sequence of projections.

From (\ref{item:stineB}) we have that $P_{n - 1} W_n A_n W_n^* P_{n - 1} = W_{n - 1} A_{n - 1} W_{n - 1}^*$, and hence if we define the $C^*$-algebra $B = \{ x \in \BB(\KK) \mid W_n^* x W_n \in A_n, n \geq 0 \}$, then we have $W_n^* B W_n = A_n$, for all $n \geq 0$. Also, if $A_0$ is a von Neumann algebra, then so is $A_n$ for each $n \in \N$ and from this it follows easily that $B$ is also a von Neumann algebra.

We define the unital $*$-endomorphism $\alpha : B \to B$ (which is normal when $A_0$ is a von Neumann algebra and $\phi_0$ is normal) by the formula 
$$
\alpha(x) = \lim_{n \to \infty} W_{n + 1}\pi_{n + 1}(W_n^* x W_n) W_{n + 1},
$$ 
where the limit is taken in the strong operator topology. Note that $\alpha(P_n) = P_{n + 1} \geq P_n$. From (\ref{item:stineD}) we see that in general, the strong operator topology limit exists in $B$, and that for $x \in A_n \cong P_nA_\infty P_n$ the limit stabilizes as $\alpha( W_n x W_n^* ) =  W_{n + 1}\pi_{n + 1}( x ) W_{n + 1}^*$.

From (\ref{item:stineA}) we see that for $n \geq 0$, and $x \in A_n$ we have 
\begin{align}
P_n \alpha ( W_n x W_n^* ) P_n 
&= W_n W_n^*W_{n + 1}\pi_{n + 1}(x) W_{n + 1}^*W_n W_n^* \nonumber \\
&= W_n V_{n + 1}^* \pi_{n + 1}(x) V_{n + 1} W_n^* \nonumber \\
&= W_n \phi_n( x ) W_n^*. \nonumber
\end{align}
By induction we then see that also for $k > 1$, and $x \in A_0$ we have 
\begin{align}
P_0 \alpha^k ( W_0 x W_0^*  ) P_0
&= P_0 \alpha^{k - 1}( P_0 \alpha( W_0 x W_0^* ) P_0 ) P_0 \nonumber \\
&= P_0 \alpha^{k - 1} (W_0 \phi_0(x) W_0^* )P_0 \nonumber \\
&= W_0 \phi_0^k(x) W_0^*. \nonumber
\end{align}

By the previous Lemma we have that the central support of $P_n$ in $W_n A_n'' W_n^*$ is $P_{n + 1}$. Hence it follows that the central support of $P_0$ in $B$ is $1$. 
\end{proof}

\subsection{Poisson boundaries of u.c.p.\ maps}

If $A \subset \BB(\HH)$ is a unital $C^*$-algebra, and $\phi: A \to A$ a unital completely positive map, then a projection $p \in A$ is said to be coinvariant, if $\{ \phi^n(p) \}_n$ defines an increasing sequence of projections which strongly converge to $1$ in $\BB(\HH)$, and such that for $y \in \BB(\HH)$ we have $y \in A$ if and only if $\phi^n(p) y \phi^n(p) \in A$ for all $n \geq 0$. Note that for $n \geq 0$, $\phi^n(p)$ is in the multiplicative domain for $\phi$, and is again coinvariant. We define $\phi_p: pAp \to pAp$ to be the map $\phi_p(x) = p \phi(x) p$, then $\phi_p$ is normal unital completely positive. Moreover, we have that $\phi_p^k(x) = p \phi^k(x) p$ for all $x \in pAp$, which can be seen by induction from 
$$
p \phi^k(x) p = p \phi^{k - 1}(p) \phi^k(x) \phi^{k - 1}(p) p = p \phi^{k - 1}( \phi_p(x) ) p.
$$

\begin{thm}[Prunaru \cite{prunaru}]\label{thm:izumipb}
Let $A \subset \BB(\HH)$ be a unital $C^*$-algebra, $\phi: A \to A$ a unital completely positive map, and $p \in A$ a coinvariant projection. Then the map $\mathcal P: \Har(A, \phi) \to \Har(pAp, \phi_p)$ given by $\mathcal P(x) = pxp$ defines a completely positive isometric surjection, between $\Har(A, \phi)$ and $\Har(pAp, \phi_p)$.

Moreover, if $A$ is a von Neumann algebra and $\phi$ is normal, then $\mathcal P$ is also normal.
\end{thm}
 \begin{proof}
First note that $\mathcal P$ is well-defined since if $x \in \Har(A, \phi)$ we have 
\[
\phi_p(pxp) = p \phi(p) x \phi(p) p = p x p.
\] 
Clearly $\mathcal P$ is completely positive (and normal in the case when $A$ is a von Neumann algebra and $\phi$ is normal). 

To see that it is surjective, if $x \in \Har(pAp, \phi_p)$, then consider the sequence $\phi^n(x)$. For each $m, n \geq 0$, we have 
$$
\phi^m(p) \phi^{m + n}(x) \phi^m(p) = \phi^{m} ( p\phi^n(x) p ) = \phi^m( \phi_p^n(x)) = \phi^m(x).
$$
It follows that $\{ \phi^n(x) \}_n$ is eventually constant for any $\xi$ in the range of $\phi^m(p)$ for any $m$. Since $\{ \phi^n(x) \}_n$ is uniformly bounded and $\{\phi^n(x) \xi \}_n$ converges for a dense subset of $\xi \in \HH$ we then have that $\{ \phi^n(x) \}_n$
converges in the strong operator topology to an element $y \in \BB(\HH)$ 
such that $\phi^m(p) y \phi^m(p) = \phi^m(x)$ for each $m \geq 0$. Consequently we have $y \in A$.

In particular, for $m = 0$ we have $p y p = x$. To see that $y \in \Har(A, \phi)$ we use that for all $z \in A$ we have the strong operator topology limit 
$$
\lim_{n \to \infty} \phi( \phi^n(p) z \phi^n(p) ) = \lim_{n \to \infty}\phi^{n + 1}(p) \phi(z) \phi^{n + 1}(p) = \phi(z),
$$ 
and hence 
\[
\phi(y) = \lim_{m \to \infty} \phi(\phi^m(p) y \phi^m(p)) = \lim_{m \to \infty} \phi^{m + 1}(x) = y.
\]
Thus $\mathcal P$ is surjective, and since $\phi^n(p)$ converges strongly to $1$, and each $\phi^n(p)$ is in the multiplicative domain of $\phi$, it follows that if $x \in \Har(A, \phi)$, then $\phi^n(pxp)$ converges strongly to $x$ and hence
$$
\| x \| = \lim_{n \to \infty} \| \phi^n(pxp) \| \leq \| pxp \| \leq \| x \|.
$$
Thus, $\mathcal P$ is also isometric.
\end{proof}
 
 \begin{cor}[Izumi \cite{izumi2}]\label{cor:poissonb}
 Let $A$ be a unital $C^*$-algebra, and $\phi: A \to A$ a unital completely positive map. Then there exists a $C^*$-algebra $B$ and a completely positive isometric surjection $\mathcal P: B \to \Har(A, \phi)$. 
 
Moreover $B$ and $\mathcal P$ are unique in the sense that if $\tilde B$ is another $C^*$-algebra, and $\mathcal P_0: \tilde B \to \Har(A, \phi)$ is a completely positive isometric surjection, then $\mathcal P^{-1} \circ \mathcal P_0$ is an isomorphism.

Also, if $A$ is a von Neumann algebra and $\phi$ is normal, then $B$ is also a von Neumann algebra and $\mathcal P$ is normal.
 \end{cor}
 \begin{proof}
Note that we may assume $A \subset \BB(\HH)$. Existence then follows by applying the previous theorem to Bhat's dilation. Uniqueness follows from \cite{choi}
 \end{proof}

\begin{cor}[Choi-Effros \cite{choieffros}]\label{cor:choieffros}
Let $A$ be a unital $C^*$-algebra and $F \subset A$ an operator system. If $E: A \to F$ is a completely positive map such that $E_{|F} = {\rm id}$, then $F$ has a unique $C^*$-algebraic structure which is given by $x \cdot y = E(xy)$. Moreover, if $A$ is a von Neumann algebra and $F$ is weakly closed, then this gives a von Neumann algebraic structure on $F$.
\end{cor}
\begin{proof}
Note that $F \subseteq \Har(A,E)$, as $E_{|F} = {\rm id}$. Since the range of $E$ is contained in $F$, we get $\Har(A, E) = F$.

When $A$ is a $C^*$-algebra this follows from Corollary~\ref{cor:poissonb} since $\Har(A, E) = F$. Also note that since $E^n = E$ it follows from the proof of Theorem~\ref{thm:izumipb} that the product structure coming from the Poisson boundary is given by $x \cdot y = E(xy)$.

If $A$ is a von Neumann algebra and $F$ is weakly closed, then $F$ has a predual $F_\perp = \{ \varphi \in A_* \mid \varphi(x) = 0, \mbox{ for all } x \in F \}$ and hence $A$ is isomorphic to a von Neumann algebra by Sakai's theorem. 
\end{proof}

\begin{prop}
Let $A$ be an abelian $C^*$-algebra and $\phi: A \to A$ a normal unital completely positive map. Then the Poisson boundary of $\phi$ is also abelian.
\end{prop}
\begin{proof}
Let $B$ be the Poisson boundary of $\phi$, and let $\mathcal P: B \to \Har(A, \phi)$ be the Poisson transform. If $C$ is a $C^*$-algebra and $\psi: C \to B$ is a positive map, then $\mathcal P \circ \psi : C \to \Har(A, \phi) \subset A$ is positive, and since $A$ is abelian it is then completely positive.
Hence, $\psi$ is also completely positive. Since every positive map from a $C^*$-algebra to $B$ is completely positive it then follows that $B$ is abelian.
\end{proof}

\begin{examp}\label{examp:groupboundary}
Let $\Gamma$ be a discrete group and $\mu \in \Prob(\Gamma)$ a probability measure on $\Gamma$ such that the support of $\mu$ generates $\Gamma$. Then on $\ell^\infty \Gamma$ we may consider the normal unital (completely) positive map $\phi_\mu$ given by $\phi_\mu(f) = \mu * f$, where $\mu * f$ is the convolution $(\mu * f) (x ) = \int f(g^{-1} x) \, d\mu(g)$. Then $\Har(\mu) = \Har(\ell^\infty\Gamma, \phi_\mu)$ has a unique von Neumann algebraic structure which is abelian by the previous proposition. Notice that $\Gamma$ acts on $\Har(\mu)$ by right translation, and since this action preserves positivity it follows from \cite{choi} that $\Gamma$ preserves the multiplication structure as well.  

Since the support of $\mu$ generates $\Gamma$, for a non-negative function $f \in \Har(\mu)_+$, we have $f(e) = 0$ if and only if $f = 0$. Thus we obtain a natural normal faithful state $\varphi$ on $\Har(\mu)$ which is given by $\varphi(f) = f(e)$. 

Since $\varphi$ is $\Gamma$-equivariant, this extends to a normal u.c.p.\ map $\tilde \varphi : \ell^\infty \Gamma \rtimes \Gamma \to \ell^\infty \Gamma \rtimes \Gamma$ such that $\tilde \varphi_{L\Gamma} = {\rm id}$. Note that $\ell^\infty \Gamma \rtimes \Gamma \cong \mathcal B(\ell^2 \Gamma)$. It is an easy exercise to see that the Poisson boundary of $\tilde \varphi$ is nothing but the crossed product $\Har(\mu) \rtimes \Gamma$.
\end{examp}

\def\cprime{$'$}
\providecommand{\bysame}{\leavevmode\hbox to3em{\hrulefill}\thinspace}
\providecommand{\MR}{\relax\ifhmode\unskip\space\fi MR }
\providecommand{\MRhref}[2]{%
  \href{http://www.ams.org/mathscinet-getitem?mr=#1}{#2}
}
\providecommand{\href}[2]{#2}

\end{document}